\documentclass[12pt]{amsart}
\pdfoutput=1
\usepackage{amssymb}
\usepackage{amsmath}
\usepackage{amsfonts}
\usepackage[usenames]{color}
\usepackage{graphicx}
\usepackage{array}
\usepackage{psfrag}
\usepackage{color}
\usepackage{ulem}
\usepackage{fancyhdr}
\usepackage[table,xcdraw]{xcolor}
\usepackage{multirow}
\usepackage{hyperref}

\makeatletter
 
 \@addtoreset{equation}{section}
\makeatother

\newtheorem{theorem}{Theorem}
\newtheorem{lemma}[theorem]{Lemma}
\newtheorem{proposition}[theorem]{Proposition}

\theoremstyle{definition}

\theoremstyle{remark}
\newtheorem{remark}[theorem]{Remark}

\numberwithin{equation}{section}

\newcommand{\dfn}[1]{\textit{#1}}

\newcommand{\ty}{\nabla Y}
\newcommand{\yt}{Y \nabla}

\newcommand{\etf}{G_{11,35}} 
\newcommand{\tts}{G_{10,26}} 
\newcommand{\tth}{G_{10,30}} 

\begin{document}

\title[IK graphs with IL simple minors]{Intrinsically knotted graphs with linklessly embeddable simple minors}

\date{\today}
\author[Mattman, Naimi, Pavelescu, and Pavelescu]{
Thomas W.~Mattman, Ramin Naimi, Andrei Pavelescu, and Elena Pavelescu
}

\begin{abstract}
It has been an open question whether the deletion or contraction
of an edge in an intrinsically knotted graph
always yields an intrinsically linked graph.
We present a new intrinsically knotted graph that shows
the answer to both questions is no.
\end{abstract}

\maketitle
\rhead{IK}

\section{Introduction}
\label{introduction}

A graph is  \dfn{intrinsically knotted} (\dfn{intrinsically linked}) if every embedding of it in $S^3$ contains a nontrivial knot (2-component link).
We abbreviate intrinsically knotted (linked) as IK (IL), 
and not intrinsically knotted (linked) as nIK (nIL).
Robertson, Seymour, and Thomas ~\cite{RST} showed that every IK graph is IL. 
It is also known that coning one vertex over an IL graph yields an IK graph.
(This is shown by combining the work of \cite{F2}, \cite{RST}, and \cite{Sa}.)
However, it has been difficult to make the relationship between IK and IL graphs stronger. For example, Adams~\cite{A}
asked if deleting a vertex from an IK graph always yields an IL graph, 
but Foisy~\cite{F} provided a counterexample. 
Deleting a vertex from a graph also deletes all edges incident to that vertex. 
So it might seem more likely that deleting, or contracting, a single edge of an IK 
graph should leave it IL. 
Naimi, Pavelescu, and Schwartz~\cite{NPS} 
tried to show that this is the case when the edge belongs to a $3$-cycle,
but their proof contained an error (which we will describe in Section~\ref{erratum}).
They also asked if deleting or contracting an edge in an IK graph always yields an IL graph.
We verify (using a computer program) that the answer to this question is yes for graphs of order at most 9;
but we show that in general the answer is no. 
We present an IK graph $\etf$ of order 11 and size 35 
with edges $e$ and $f$
such that neither $\etf-e$ (edge deletion) nor $\etf /f$ (edge contraction) is IL. 
We argue that $\etf$ is a minimal order example of an IK graph 
that yields a nIL graph by deleting one edge;
and that ten is the smallest order for an IK graph 
that yields a nIL graph by contracting one edge.
The graph $\etf$ is also a counterexample to the main result of ~\cite{NPS}.

Graphs that are IK but yield a nIL graph 
by deleting one vertex or edge or by contracting one edge
are intriguing from the perspective of Colin de Verdi\`ere's graph invariant $\mu$. 
This is an integer-valued graph invariant that is difficult to compute in general;
its value is known only for certain classes of graphs with ``nice'' topological properties.
For example, for any graph $G$, 
$\mu(G) \leq 3$ if and only if $G$ is planar~\cite{dV},
and $\mu(G) \leq 4$ if and only if $G$ is nIL~\cite{HLS}. 

An important open question is how to characterize graphs $G$ with $\mu(G) \le  5$. 
Even though many known minor minimal IK (MMIK) graphs have $\mu$-invariant 6,
intrinsic knottedness is not the answer.
A \dfn{minor} of a graph $G$ is a graph obtained 
by contracting zero or more edges in a subgraph of $G$.
We'll say an \dfn{edge deletion minor} (\dfn{edge contraction minor}) of $G$
is a graph obtained by deleting (contracting) exactly one edge of $G$.
Both are called \dfn{simple minors} of $G$.
As we explain in Section~\ref{sec:mu5IK}, if an IK graph $G$ has a nIL
simple minor then $\mu(G) = 5$.
Thus, our graph $\etf$, together with other IK graphs obtained from it (as described in Section~\ref{sec:mu5IK}), 
join Foisy's graph as new examples of IK graphs with $\mu$-invariant 5.
These examples show that $\mu(G) \leq 5$ is not equivalent to $G$ being nIK.

In the next section we describe the graph $\etf$ and we show it is IK and minor minimal for that property
in Sections 3 and 4, respectively. In Section 5 we make some observations about the Colin de Verdi{\`e}re invariant
and prove that 10 is the least order for an IK graph with an edge contraction minor that is IL. Section 6 goes over
the error in~\cite{NPS} and we conclude with an Appendix that provides edge lists for three
graphs discussed in the paper. 

To complete this introduction, we provide several definitions.
A graph $G$ is \dfn{$n$-apex} if one can delete $n$ vertices from $G$ to obtain a planar graph;
$G$ is \dfn{apex} if it is 1-apex, and 0-apex is a synonym for planar. 
A graph $G$ is \dfn{minor minimal} with respect to a property if
$G$ has that property but no minor of it has that property.
The complete graph on $n$ vertices is denoted as $K_n$.
$V(G)$ and $E(G)$ denote, respectively, the vertex set and the edge set of $G$.
 A graph $G$ is the \textit{clique sum} of two subgraphs $G_1$ and $G_2$ over $K_n$ if $V(G)=V(G_1)\cup V(G_2)$, $E(G)=E(G_1)\cup E(G_2)$, and the subgraphs induced in $G_1$ and $G_2$ by $V(G_1)\cap V(G_2)$ are both isomorphic to $K_n$.
 We use the notation $G=G_1\oplus_{K_n}G_2$.  
 The \dfn{$\ty$-move} and \dfn{$\yt$-move} are defined as shown in Figure~\ref{fig-TYmove}.
The \dfn{family} of a graph $G$ is the set of all graphs obtained from $G$
by doing zero or more $\ty$ and $\yt$ moves.
The Petersen family of graphs is the family of the Petersen graph
(which is also the family of $K_6$).

 \begin{figure}[ht]
 \centering
 \includegraphics[width=70mm]{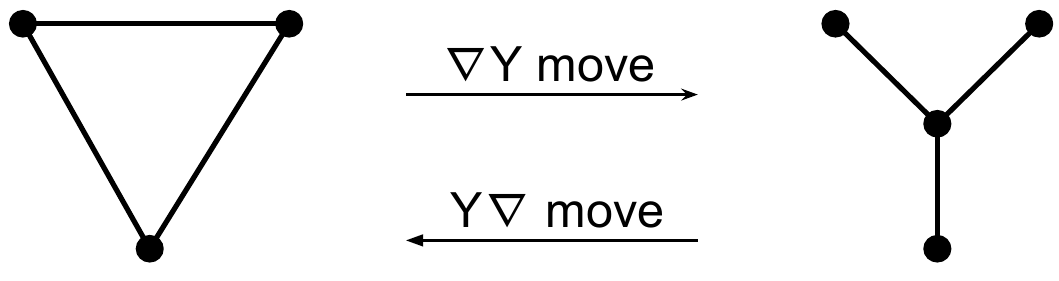}
 \caption{$\ty$ and $\yt$ moves.
 \label {fig-TYmove}
 }
\end{figure}


\section{The graph $\etf$}
\label{graph11}

We describe a sequence of graphs and graph operations used to construct the graph $\etf$.
Let $H$ denote the graph in Figure ~\ref{apex1}(a). Deleting the vertex labeled 4, one obtains the maximal planar graph $H'$, depicted in Figure~\ref{apex1}(b). 
This implies that $H$ is an apex graph; thus it is nIL by \cite{Sa}. 
Similarly, the graph $K$ shown in Figure ~\ref{apex2}(a) is nIL
since deleting vertex 5 from $K$ yields a maximal planar graph, as in Figure~\ref{apex2}(b). 

 \begin{figure}[ht]
\begin{center}
\begin{picture}(350, 140)
\put(0,0){\includegraphics[width=5in]{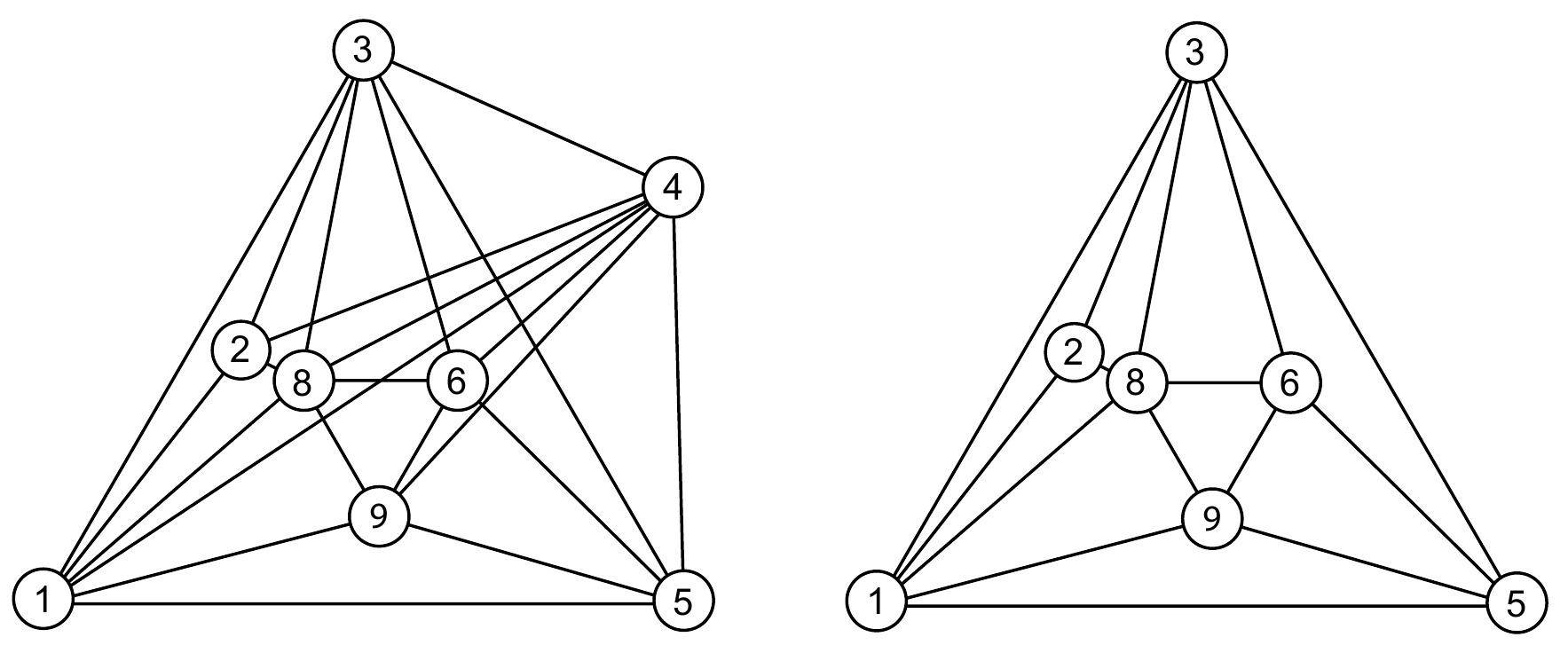}}
\end{picture}
\caption{(Left)  $H$ is apex.  (Right)  $H'$ is maximal planar. }
\label{apex1}
\end{center}
\end{figure} 

 \begin{figure}[ht]
\begin{center}
\begin{picture}(350, 140)
\put(0,0){\includegraphics[width=5in]{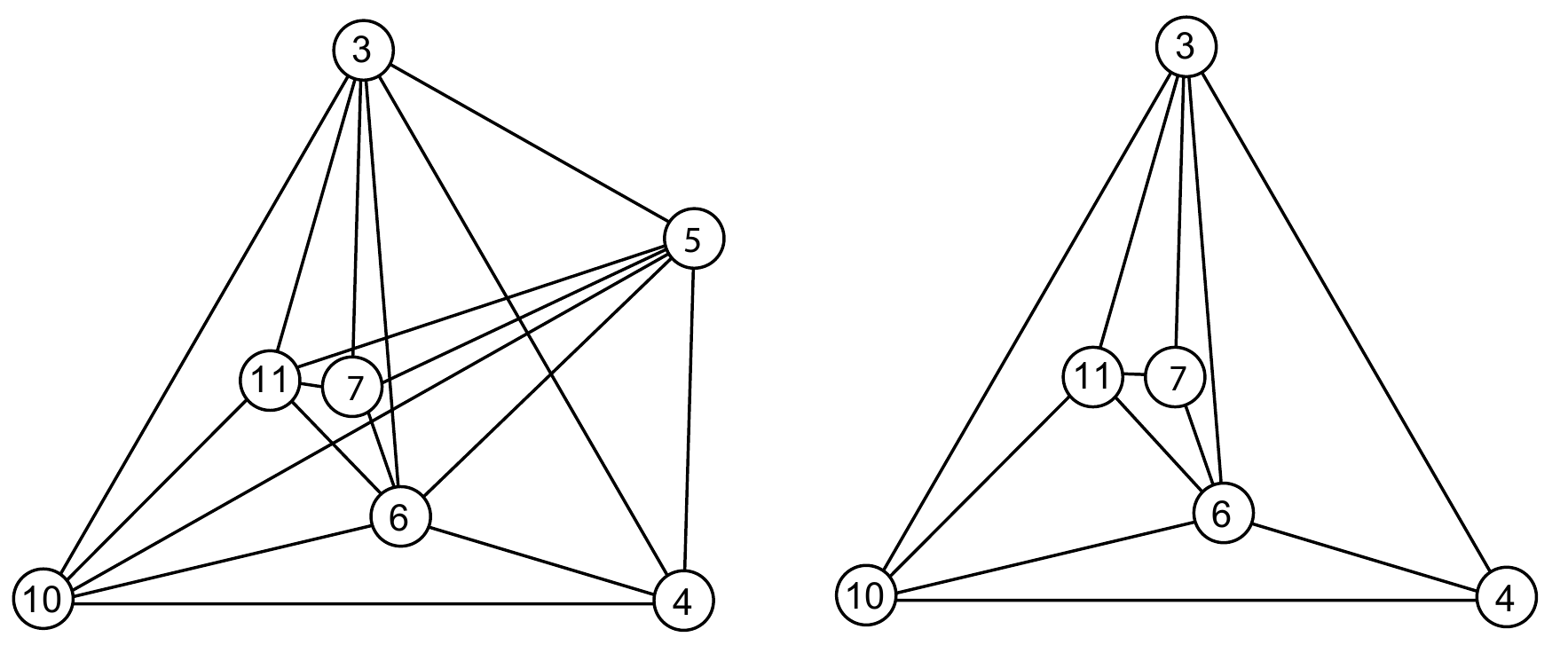}}
\end{picture}
\caption{(Left)  $K$ is apex.  (Right)  $K'$ is maximal planar. }
\label{apex2}
\end{center}
\end{figure}

Notice that deleting the vertices 3, 4, 5, and 6 from both $H$ and $K$ produces connected subgraphs.
So, by Lemma~14 of \cite{NPP}, the clique sum of $H$ and $K$ over the $K_4$ induced by $\{3, 4, 5, 6\}$ is a nIL graph, denoted by $M$ and depicted in Figure ~\ref{maxnIL}.

 \begin{figure}[htpb!]
\begin{center}
\begin{picture}(250, 160)
\put(0,0){\includegraphics[width=3.5in]{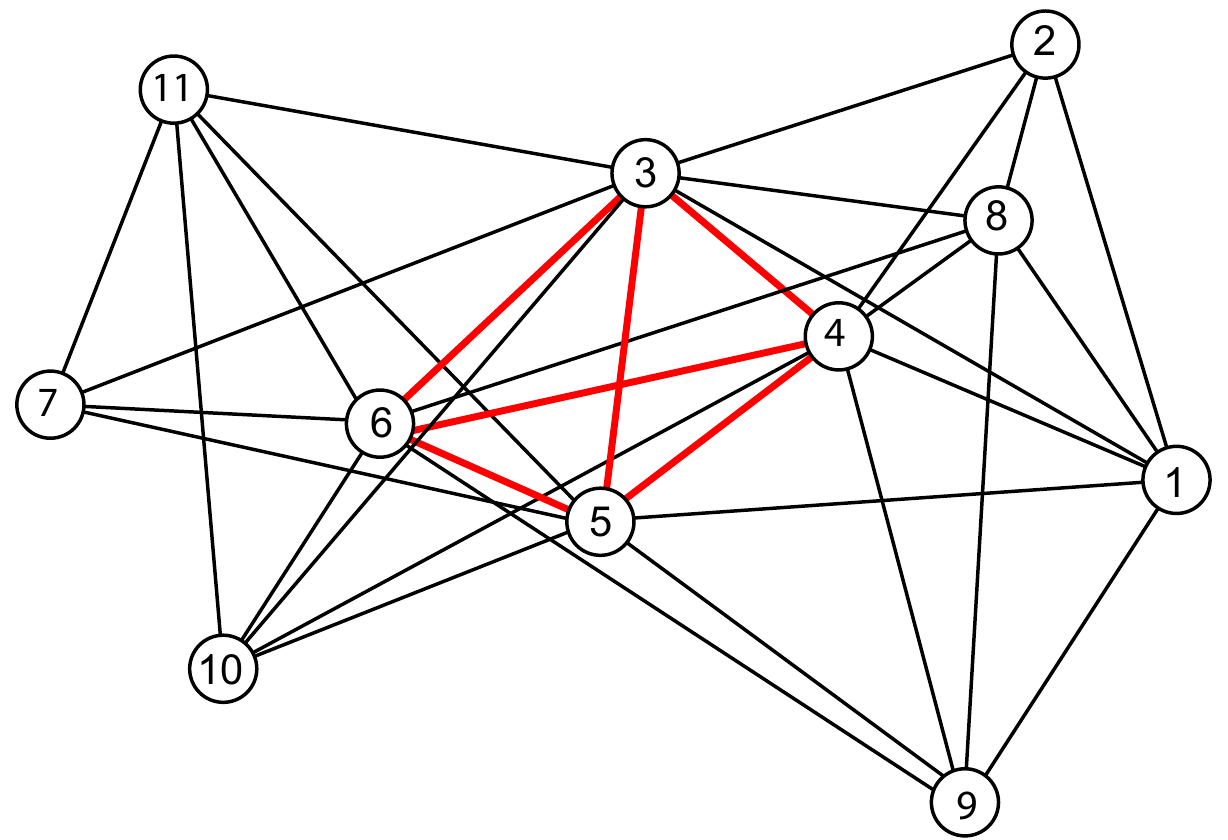}}
\end{picture}
\caption{$M\simeq H\oplus_{K_4}K$ }
\label{maxnIL}
\end{center}
\end{figure} 

The graph $\etf$ is obtained by adding the edge $(2,11)$ to the nIL graph $M$ 
(see Figure ~\ref{ik11}).
We prove in Section ~\ref{Order 10 IK minor} that $\etf$ is IK.
We have thus obtained an IK graph that has a nIL edge deletion minor.
Furthermore, since the edge $(2,11)$ is in a 3-cycle in $\etf$,
this also gives a counterexample to the main result of \cite{NPS}.
Notice that contracting the edge $(2,3)$ in $\etf$
yields a graph that is a minor of $M$, and therefore nIL. 
Hence, $\etf$ also has a nIL edge contraction minor. 
The edge list of $\etf$ is given in the Appendix~\ref{appndx}.

 \begin{figure}[htpb!]
\begin{center}
\begin{picture}(250, 190)
\put(0,0){\includegraphics[width=3.8in]{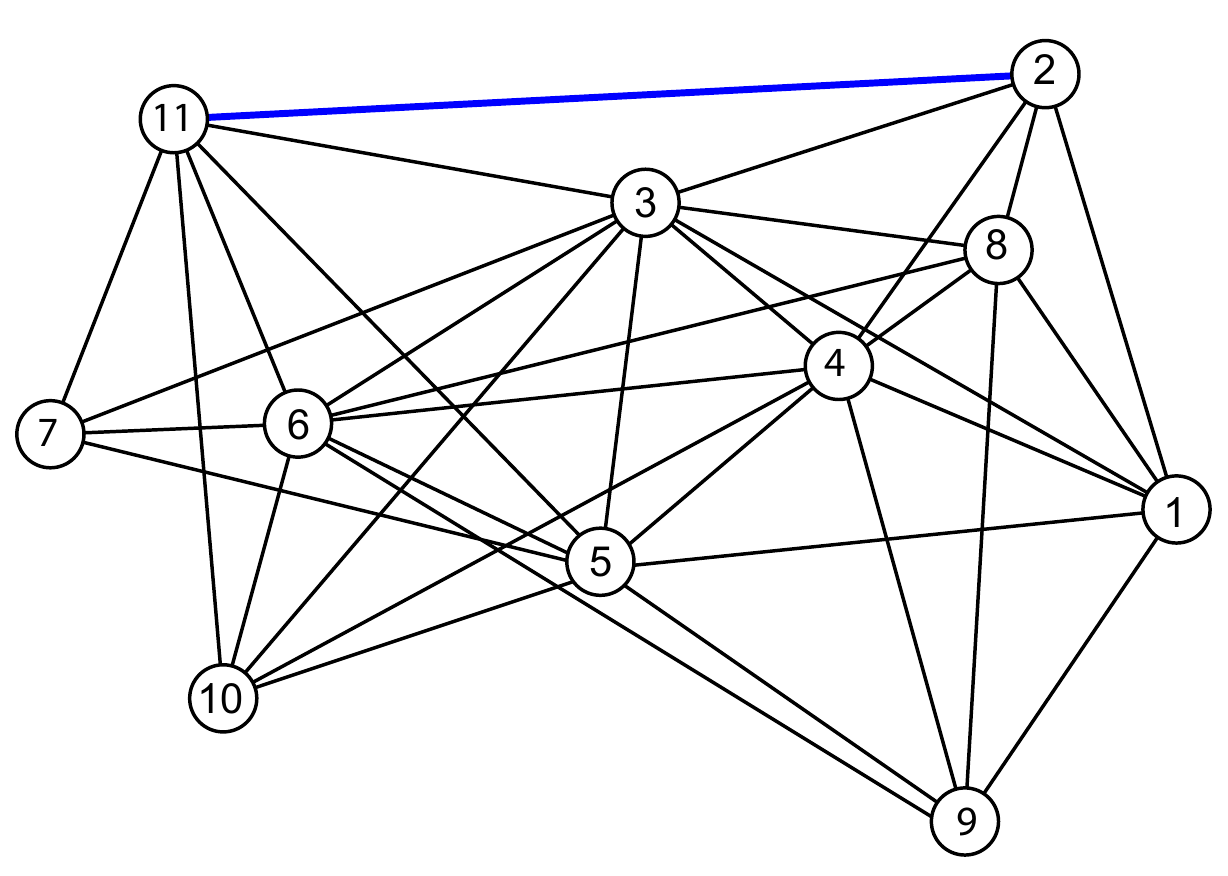}}
\end{picture}
\caption{The graph $\etf$. }
\label{ik11}
\end{center}
\end{figure}

 \begin{remark}
The edge $(2,3)$ in $\etf$ is triangular (i.e., it belongs to one or more triangles) .
So contracting it results in the deletion of parallel edges.
One can ask whether contracting a non-triangular edge in an IK graph can result in a nIL graph.
The answer is yes:
In $\etf$, if we do a $\ty$ move on the triangle with vertices $2,3,11$,
we obtain a new IK graph $G'$ with a new vertex, denoted $x$.
Contracting the edge $(x,3)$ (which is non-triangular) in $G'$ 
yields a graph isomorphic to $\etf - (2,11)$, which is nIL.
\end{remark}

 \begin{remark}
 The graph $\etf$ is a minimal order IK graph with a nIL edge deletion minor. 
To verify this, we took every maxnIL graph of order 10 (there are 107 of them~\cite{NOPP}),
and checked (with computer assistance) that adding one edge to it never yields an IK graph.  
However, 11 is not the smallest order of an IK graph that has a nIL edge contraction minor. 
The graph $\tth$ depicted in Figure ~\ref{mmik10} is a minor minimal IK graph of order 10. 
Contracting the edge $(2,6)$ gives the nIL minor in Figure ~\ref{G1030ce}(a).
This graph is nIL since adding the edge $(8,9)$ produces a graph isomorphic to the clique sum, over the $K_4$ subgraph induced by $\{2,3,8,9\}$, 
of $K_5$ and a subgraph isomorphic to $H$ introduced in Figure \ref{apex1}.
By the following proposition, $\tth$ is
a minimal order IK graph with a nIL edge contraction minor. 
In Section~\ref{sec:mu5IK}, we show that
$\tth$ is also an example of a MMIK graph with $\mu$-invariant 5.

 \end{remark}

 \begin{figure}[htpb!]
\begin{center}
\begin{picture}(250, 190)
\put(0,0){\includegraphics[width=4in]{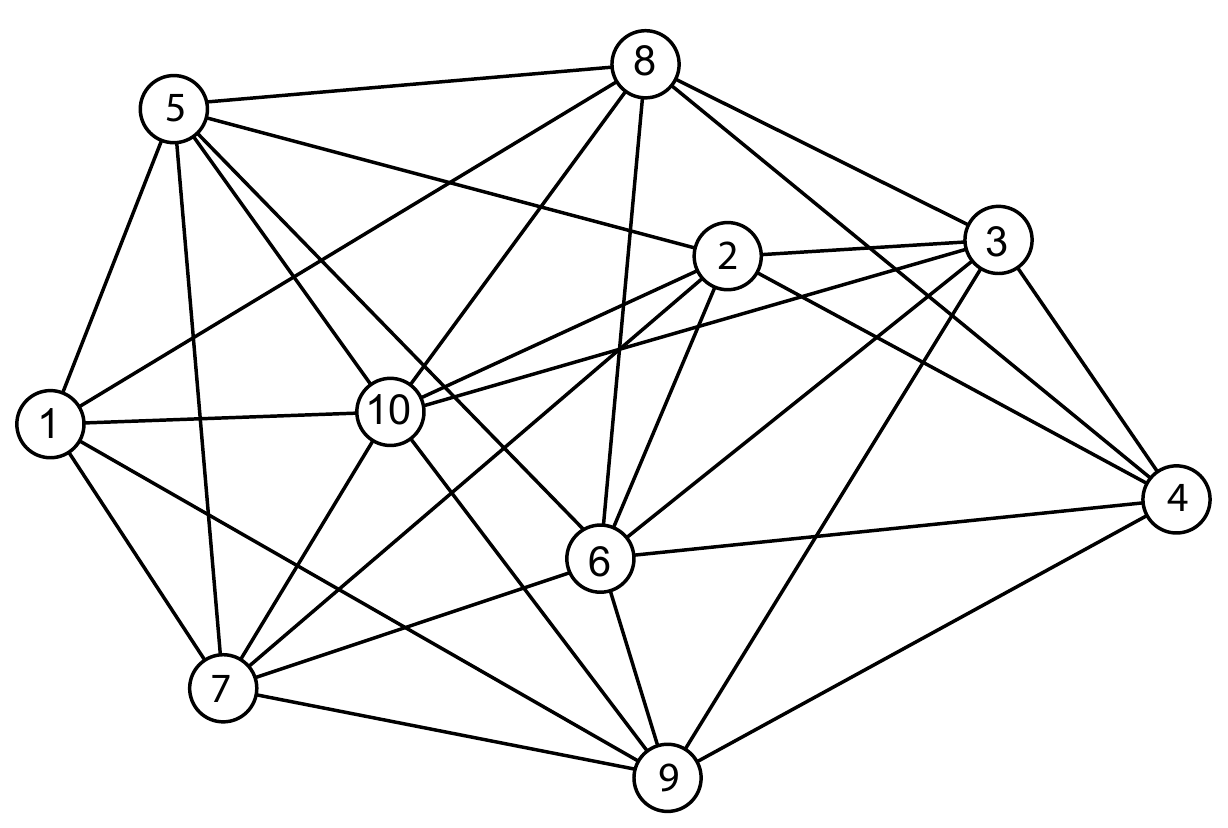}}
\end{picture}
\caption{Graph $\tth$ }
\label{mmik10}
\end{center}
\end{figure}

 \begin{figure}[htpb!]
\begin{center}
\begin{picture}(400, 140)
\put(0,0){\includegraphics[width=5.5in]{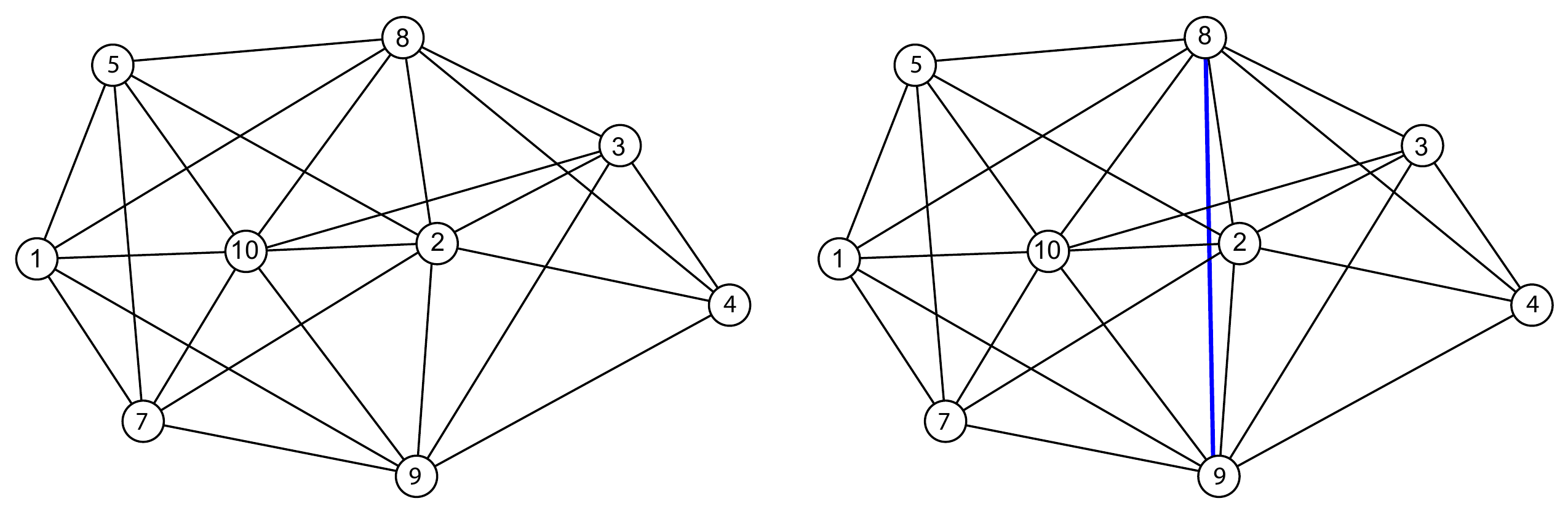}}
\end{picture}
\caption{(a) The contraction minor of $\tth$.\,\,\,\,\, (b) $H\oplus_{K_4}K_5$. }
\label{G1030ce}

\end{center}
\end{figure} 

\begin{proposition}
\label{prop:order10}
Ten is the smallest order for an IK graph which admits a nIL edge contraction minor.
\end{proposition}

We defer the proof to Section~\ref{sec:mu5IK}.
 
 
\section{$\etf$ is IK}
\label{Order 10 IK minor}

We prove $\etf$ is IK by 
showing that the graph $\tts$  in Figure~\ref{Graph10} is an IK minor of $\etf$. 
(In fact, $\tts$ is MMIK; we show this in the next section.)

 \begin{figure}[htpb!]
\begin{center}
\begin{picture}(240, 170)
\put(0,0){\includegraphics[width=3.2in]{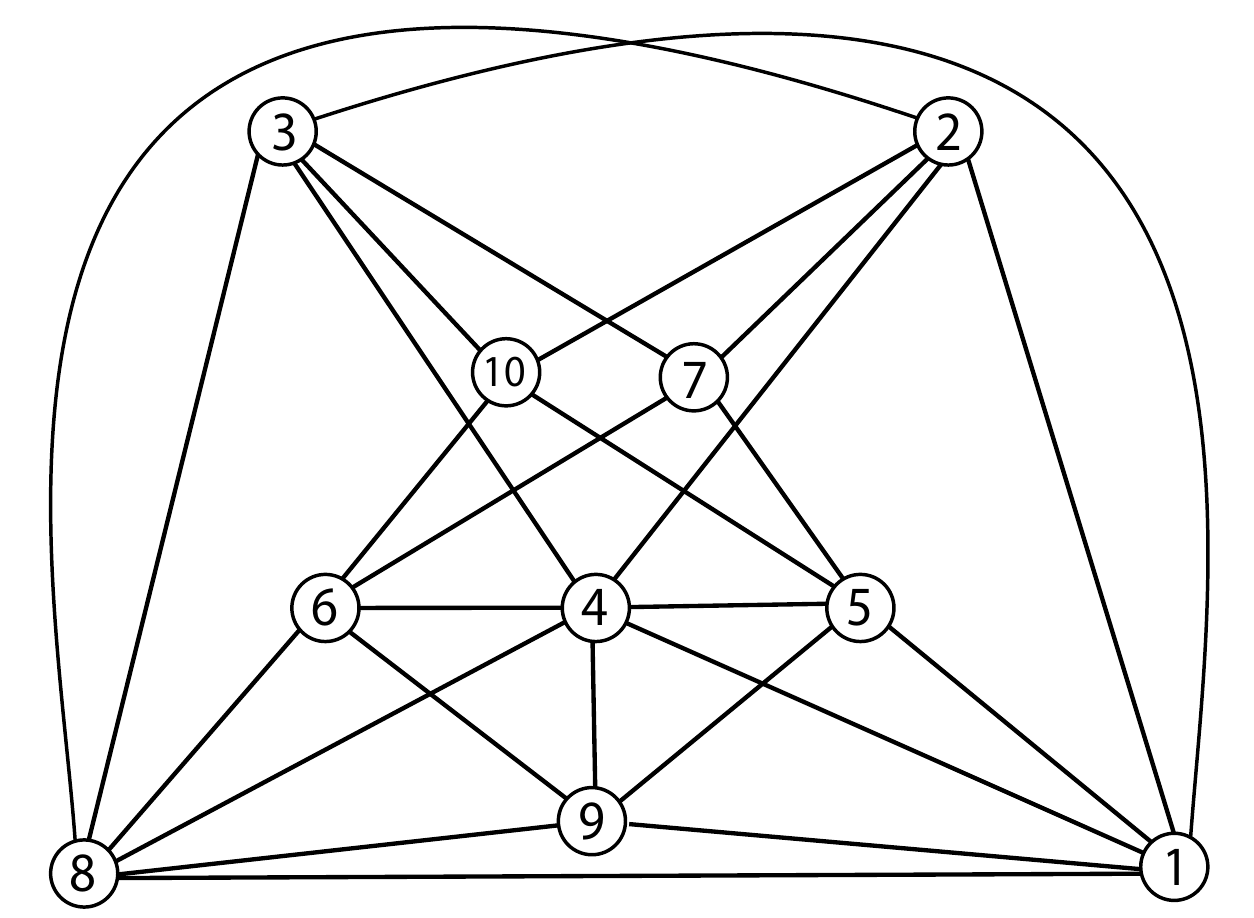}}
\end{picture}
\caption{$\tts$}
\label{Graph10}
\end{center}
\end{figure} 

The graph $\tts$ is obtained from $\etf$ by contracting the edge (2,11) and deleting the edges (2,3), (2,5), (2,6), (3,5), (3,6), (4,10),  and (5,6).

To prove $\tts$ is IK, we use the technique developed by Foisy in \cite{F2}, 
which we explain below.
The $D_4$ graph is the (multi)graph shown in Figure~\ref{fig-D4}.
A \dfn{double-linked $D_4$} is a $D_4$ graph embedded in $S^3$
such that each pair of opposite 2-cycles ($C_1 \cup C_3$, and $C_2 \cup C_4$) has odd linking number.
The following lemma was proved by Foisy~\cite{F2};
a more general version of it was proved independently 
by Taniyama and Yasuhara~\cite{TanYas}.

\begin{lemma}
\label{D4-Lemma}
Every double-linked $D_4$ contains a nontrivial knot.
\end{lemma}

 \begin{figure}[ht]
 \centering
 \includegraphics[width=35mm]{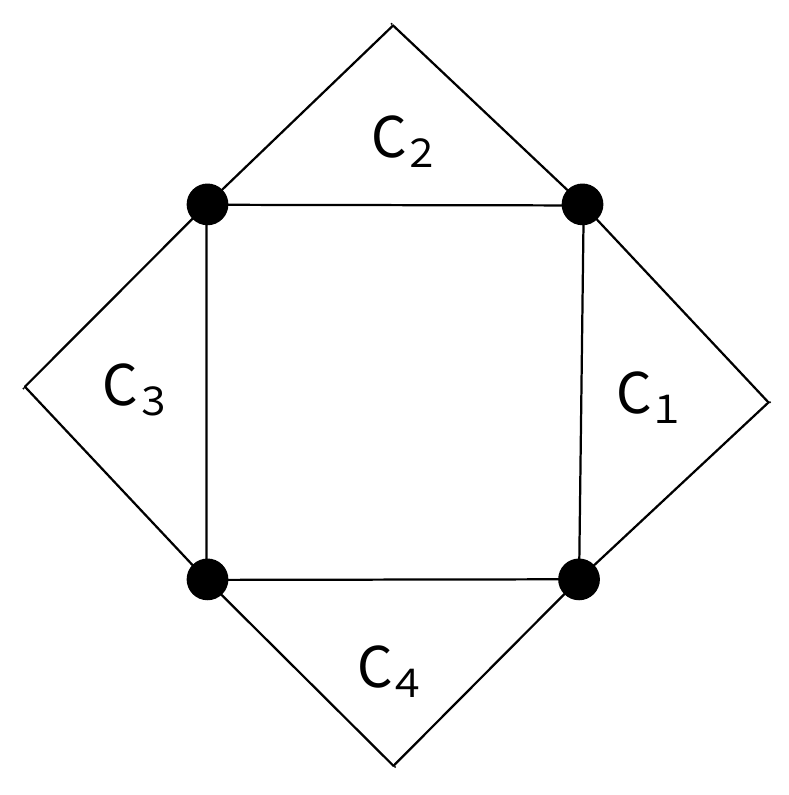}
 \caption{The $D_4$ graph.
 \label{fig-D4}
 }
\end{figure}

We will also use the following (well known and easy to prove) lemma.

\begin{lemma}
\label{lemma-split}
Suppose $\alpha$, $\beta_1$, and $\beta_2$ are simple closed curves in $S^3$ such that 
$\beta_1 \cap \beta_2$ is an arc and
$\alpha$ has odd linking number with
$(\beta_1 \cup \beta_2) \setminus \mathrm{interior}(\beta_1 \cap \beta_2)$.
Then $\alpha$ has odd linking number with $\beta_1$ or $\beta_2$. 
\end{lemma}

\begin{theorem}
The graph $\tts$ in Figure \ref{Graph10} is IK.
\end{theorem}
\begin{proof}

We shall prove that every embedding of $\tts$ 
has a double-linked $D_4$ minor.
It then follows from Lemma~\ref{D4-Lemma} that $\tts$ is IK.
For the remainder of this proof, we will say 
two disjoint simple closed curves $\alpha$ and $\beta$ in $S^3$ are \dfn{linked}, 
or $\alpha$ \dfn{links} $\beta$, 
if $\alpha \cup\beta$ has odd linking number.

In $\tts$ we select the subgraphs $A$, $B$, $C$, $D$, $E$,  and $F$
shown in Figure~\ref{ABCDE} (these are not induced subgraphs).
All these subgraphs are either in the Petersen family of graphs or have minors in this family, and are therefore intrinsically linked: 
$A$ contains a $K_{3,3,1}$ minor obtained by contracting the edge (4,6);  
$B$ is isomorphic to $K^-_{4,4}$; 
$C$ and $F$ contain $K^-_{4,4}$ minors obtained by contracting the edges $(8,9)$ and $(1,9)$, respectively; 
$D$ and $E$ contain $G_7$ minors obtained by contracting the edges $(6,7)$ and $(5,7)$, respectively.


 \begin{figure}[htpb!]
\begin{center}
\begin{picture}(380, 300)
\put(0,0){\includegraphics[width=5.5in]{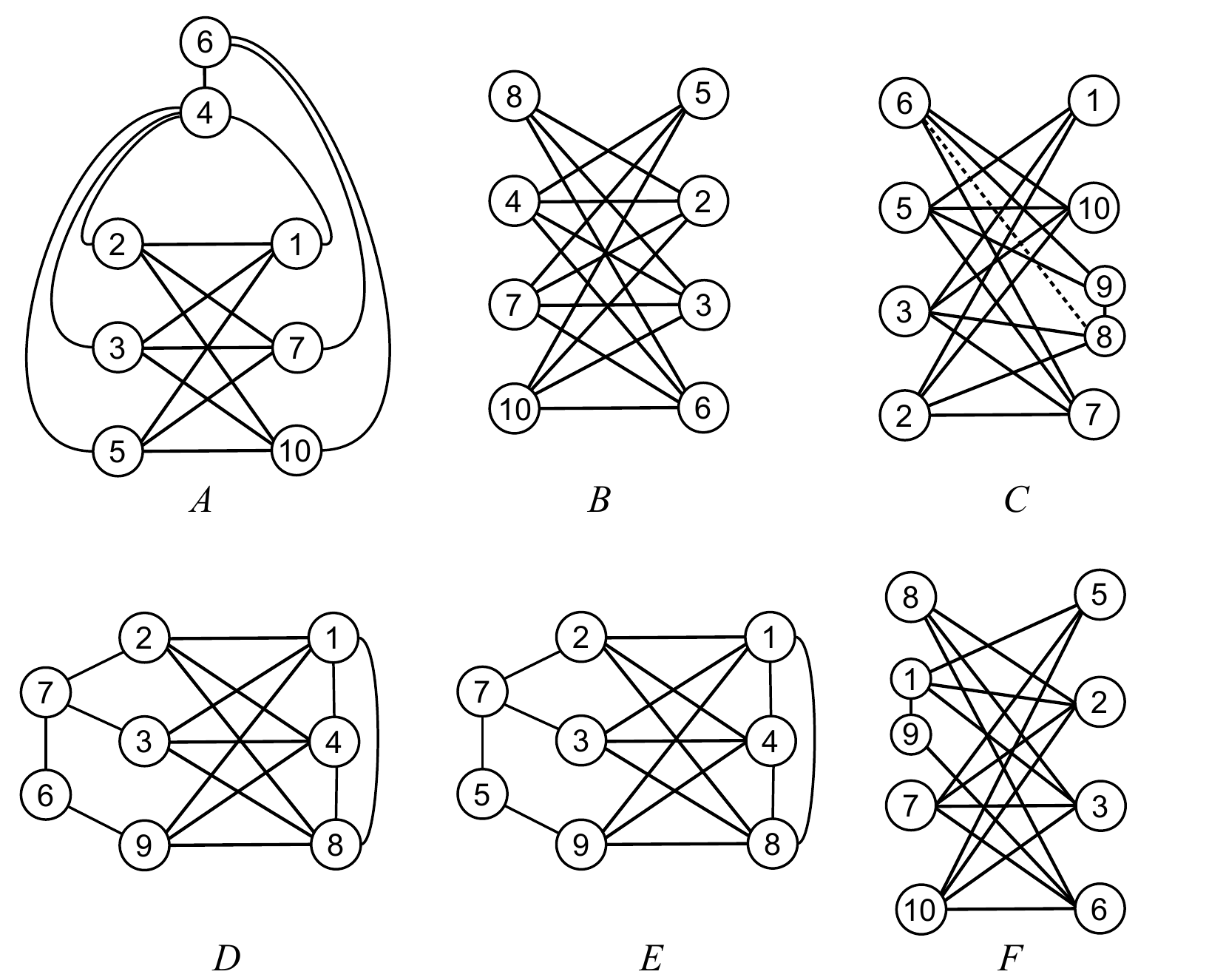}}
\end{picture}
\caption{Selected subgraphs of $\tts$.}
\label{ABCDE}
\end{center}
\end{figure} 

We organize the proof into several cases and subcases, 
according to which two cycles of each subgraph are linked.
We start with the subgraph $A$.
The vertices of $\tts$ can be partitioned into six equivalence classes up to symmetry: $\{1,8\}$, $\{2,3\}$, $\{4\}$, $\{5,6\}$, $\{7,10\}$, and $\{9\}$. 
All of these vertices except vertex 9 are in $A$.
This gives, up to symmetry, four different pairs of cycles in $A$: 
\begin{enumerate}
\item[(A1)] $(4,1,5) \cup (2, 7, 3, 10)$ 
\item[(A2)] $(4,1,2) \cup (3, 7, 5, 10)$
\item[(A3)] $(4,6,7,5) \cup (2, 1, 3, 10)$
\item[(A4)] $(4,6,7,2) \cup (3, 1, 5, 10)$
\end{enumerate}
Since $A$ is intrinsically linked,
given any embedding of $\tts$, 
we can relabel (if necessary) the vertices of $\tts$ within each equivalence class
so that at least one of these four pairs of cycles is linked.
We subdivide each of the four cases (A1)-(A4):
 (A1) is split into subcases according to which two cycles of $B$ are linked;
 (A2) according to $C$;
 (A3) according to $D$;
 and (A4) according to $B$. 
For each subcase a diagram is drawn with the non-trivial link in $A$ drawn in red. 
The two cycles in each of the subgraphs $B$ through $F$ are drawn in blue.
Each diagram contains some marked edges; contracting these marked edges in $\tts$ gives a double-linked $D_4$ minor. \\


\begin{figure}[htpb!]
\begin{center}
\begin{picture}(390, 110)
\put(0,0){\includegraphics[width=5.7in]{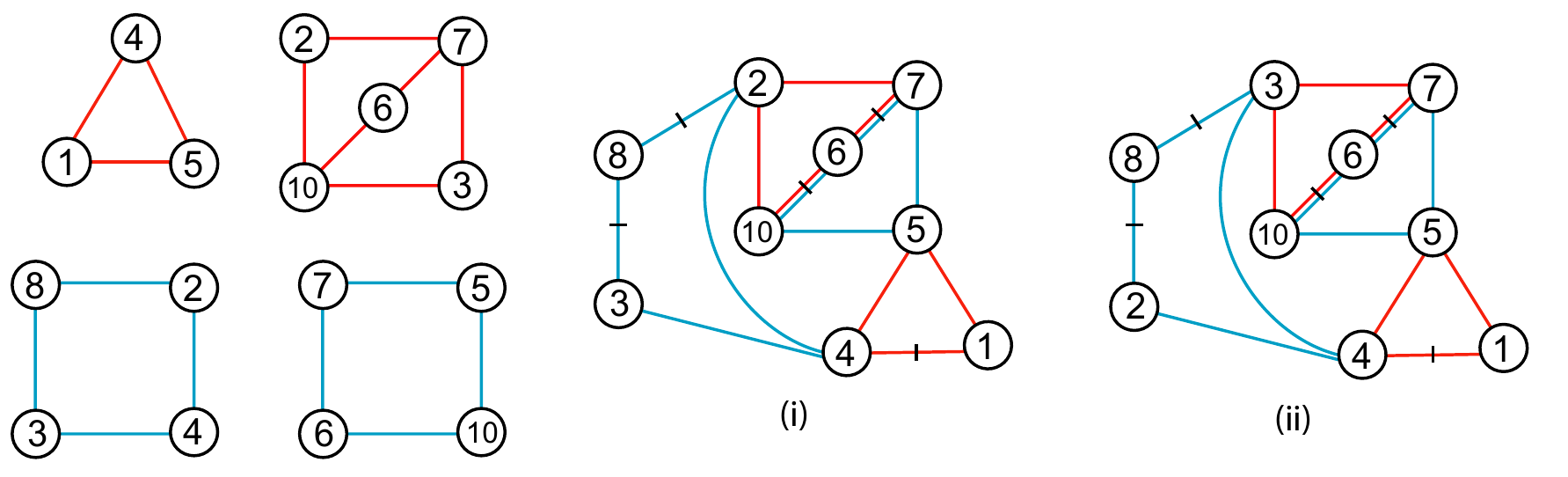}}
\end{picture}
\caption{\small Diagrams for the subcase (A1)-(B1).}
\label{a1b1}
\end{center}
\end{figure} 
\noindent \textbf{Case (A1)} Assume $(4,1,5) \cup (2, 7, 3, 10)$ is a non-trivial link of $A$.
We identify a non-trivial link in $B$ and show the existence of a double-linked $D_4$ in every subcase.  
Based on the symmetries of $\tts$, $B$ has four  different types of pairs of cycles. 
We match the link in (A1) with each of the four types of links in $B$: 
\begin{enumerate}
\item[(B1)] $(8,2,4,3) \cup (7, 5, 10, 6)$  \hspace{0.5in} (B2) $(8,2,7,3) \cup (4, 5, 10, 6)$
\item[(B3)] $(8,2,7,6) \cup (4, 5, 10, 3)$  \hspace{0.5in} (B4) $(8,2,4,6) \cup (7, 5, 10, 3)$
\end{enumerate}
\noindent \textbf{Subcase (A1)-(B1).} 
From this point forward, we abbreviate ``the cycles $X$ and $Y$ are linked'' 
as just ``$X \cup Y$''.
Assume  $(8,2,4,3) \cup (7, 5, 10, 6)$. 
Since $(4,1,5) \cup (2, 7, 3, 10)$,
by Lemma~\ref{lemma-split} we have either  (i)  $(4,1,5) \cup (2, 7, 6, 10)$ or  (ii) $(4,1,5) \cup (3, 7, 6, 10)$. 
See Figure \ref{a1b1}.


\noindent \textbf{Subcase (A1)-(B2)}. Assume  $(8,2,7,3) \cup (4, 5, 10, 6)$. See Figure \ref{a1b234} (left). \\
\noindent \textbf{Subcase (A1)-(B3)}. Assume  $(8,2,7,6) \cup (4, 5, 10, 3)$. See Figure \ref{a1b234} (center).\\ 
\noindent \textbf{Subcase (A1)-(B4)}. Assume  $(8,2,4,6) \cup (7, 5, 10, 3)$. See Figure \ref{a1b234} (right). \\


\begin{figure}[htpb!]
\begin{center}
\begin{picture}(450, 83)
\put(0,0){\includegraphics[width=6.5in]{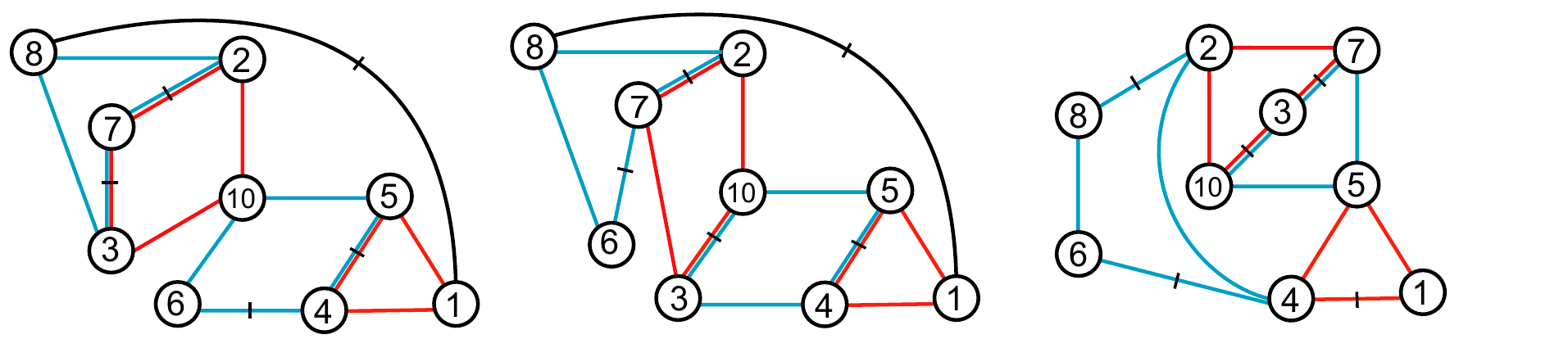}}
\end{picture}
\caption{\small Diagrams for the subcases (A1)-(B2) (left), (A1)-(B3) (center), and (A1)-(B4) (right).}
\label{a1b234}
\end{center}
\end{figure} 


\noindent \textbf{Case (A2)} Assume $(4,1,2) \cup (3, 7, 5, 10)$ is a non-trivial link of $A$.
We identify a non-trivial link in  $C$ and show the existence of a double-linked $D_4$. 
We note that vertices 8 and 9 and the edge between them act as one vertex of the $K^-_{4,4}$.
Based on the symmetries of $G$,  $C$ has four different types of pairs of cycles. 
Since in the (A2) link of $A$ vertices 2 and 3 are distinguished, they need also be distinguished within the linked cycles of $C$.
We match the link in (A2) with each link of $C$: 
\begin{enumerate}
\item[(C1)] $(6,7,2,10) \cup (1, 5, 9, 8, 3)$ \hspace{0.65in} (C4)  $(6,7,2,8,9) \cup (1, 3, 10, 5)$
\item[(C2)] $(6,7,3,10) \cup (1, 5, 9, 8, 2)$ \hspace{.65in}  (C5)  $(6,7,3,8,9) \cup (1, 2, 10, 5)$
\item[(C3)]  $(6,7,5,10) \cup (1, 2, 8, 3)$   \hspace{0.81in} (C6) $(6,7,5,9) \cup (1, 2, 10, 3)$
\end{enumerate}

\noindent \textbf{Subcase (A2)-(C1)}. Assume $(6,7,2,10) \cup (1, 5, 9, 8, 3)$. 
Since $(4,1,2) \cup (3, 7, 5, 10)$, by Lemma~\ref{lemma-split} we have either (i)  $(4,1,2) \cup (3, 7, 6, 10)$ or  (ii) $(4,1,2) \cup (5, 7, 6, 10)$.
See Figure \ref{a2c1}.


\begin{figure}[htpb!]
\begin{center}
\begin{picture}(400, 130)
\put(0,0){\includegraphics[width=5.6in]{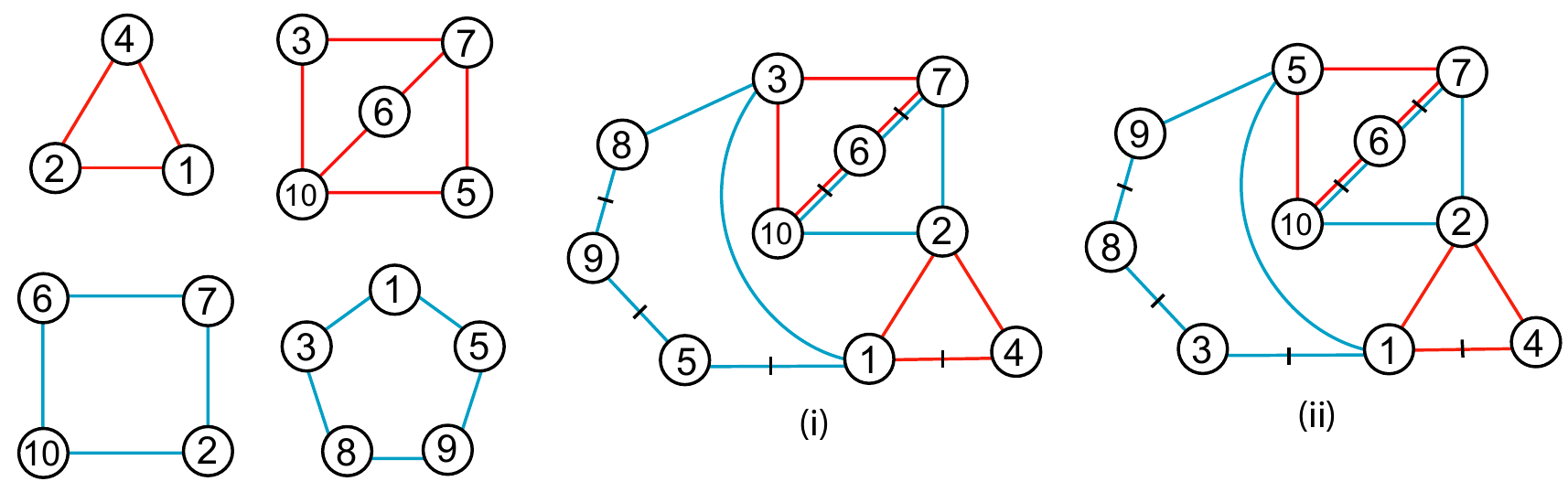}}
\end{picture}
\caption{\small Diagrams for the subcase (A2)-(C1).}
\label{a2c1}
\end{center}
\end{figure} 

\noindent \textbf{Subcase (A2)-(C2)}. Assume $(6,7,3,10) \cup (1, 5, 9, 8, 2)$. 
See Figure \ref{a2c1'} (left). 
\noindent \textbf{Subcase (A2)-(C3)}. Assume $(6,7,5,10) \cup (1, 2, 8, 3)$. 
See Figure \ref{a2c1'} (center).
\noindent \textbf{Subcase (A2)-(C4)}. Assume $(6,7,2,8,9) \cup (1, 3, 10, 5)$. 
See Figure \ref{a2c1'} (right).
\noindent \textbf{Subcase (A2)-(C5)}. Assume $(6,7,3,8,9) \cup (1, 2, 10, 5)$. 
See Figure \ref{a2c3} (left).
\noindent \textbf{Subcase (A2)-(C6)}. Assume $(6,7,3,8,9) \cup (1, 2, 10, 5)$. 
See Figure \ref{a2c3} (right).\\


\begin{figure}[htpb!]
\begin{center}
\begin{picture}(400, 105)
\put(0,0){\includegraphics[width=6in]{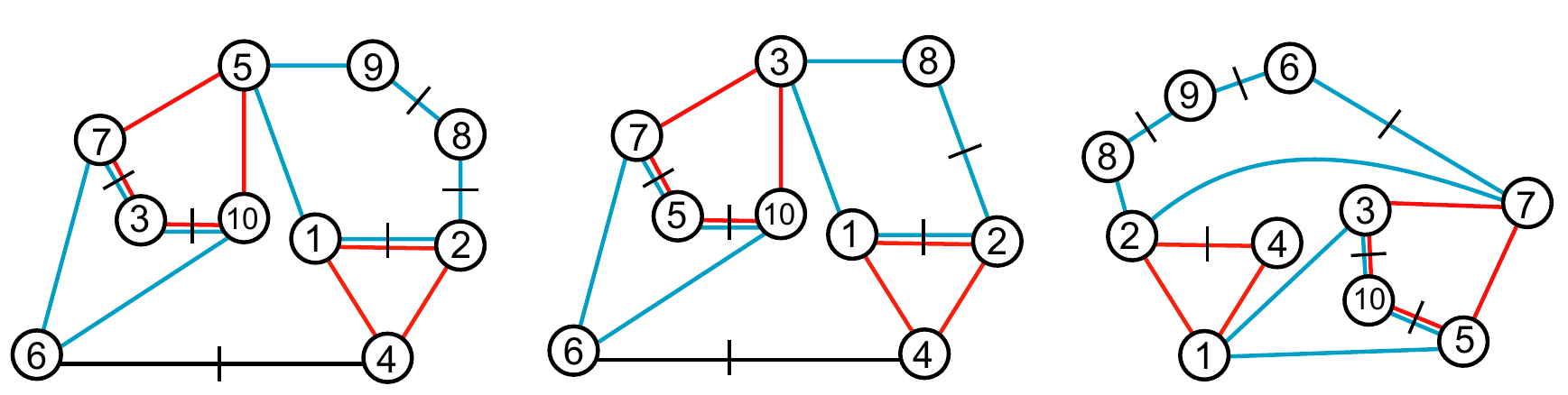}}
\end{picture}
\caption{\small Diagrams for the subcases (A2)-(C2) (left), (A2)-(C3) (center), and (A2)-(C4) (right).}
\label{a2c1'}
\end{center}
\end{figure} 


\begin{figure}[htpb!]
\begin{center}
\begin{picture}(300, 85)
\put(0,0){\includegraphics[width=4.5in]{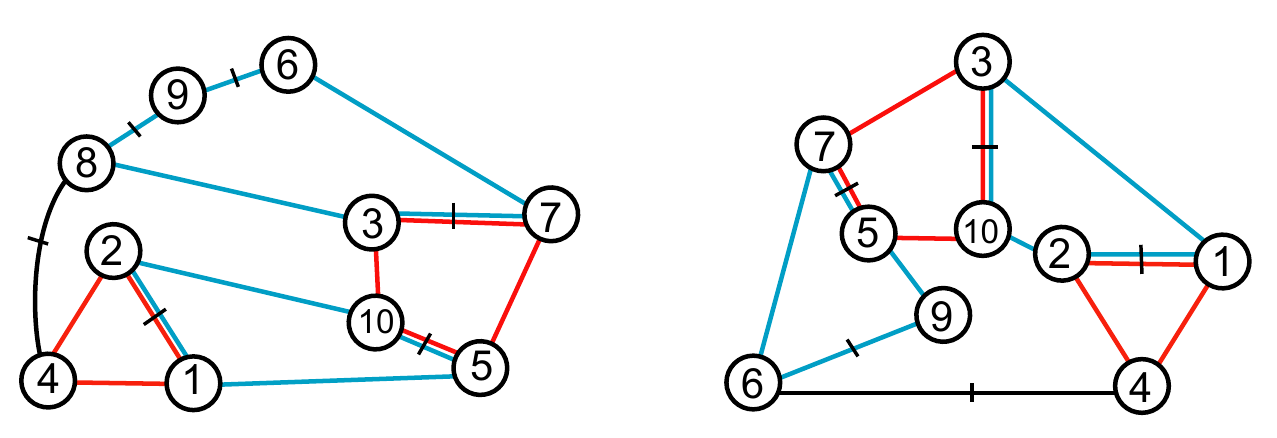}}
\end{picture}
\caption{\small Diagrams for the subcases (A2)-(C5) (left), (A2)-(C6) (right)}
\label{a2c3}
\end{center}
\end{figure}


\textbf{Case (A3)} Assume $(4,6,7,5) \cup (2, 1, 3, 10)$ is a non-trivial link of $A$.
We identify a non-trivial link in  $D$ and show the existence of a double-linked $D_4$ for all cases except one.
We then identify a non-trivial link in  $F$ and show the existence of a double-linked $D_4$ for all cases except one.
If both exceptional cases occur at the same time, the existence of a double-linked $D_4$ is shown. 

We note that  if the edge $(6,7)$ is contracted in the graph $D$, a $G_7$ graph is obtained.
Based on the symmetries of $G$,  $D$ has four different types of pairs of cycles. 
Since the (A3) link  of $A$ contains vertex 1 but does not contain vertex 8,  vertices 1 and 8 need also be distinguished within the linked cycles of $D$.
We match the link in (A3) with each link type of $D$: 
\begin{enumerate}
\item[(D1)] $(7,2,4, 3) \cup (1, 8, 9)$ \hspace{1in} (D4) $(7,2,1,9,6) \cup (4, 3, 8)$
\item[(D2)] $(7,2,1,3) \cup (4, 8, 9)$  \hspace{1in} (D5) $(7,2,8,9,6) \cup (4, 3, 1)$
\item[(D3)] $(7,2,8,3) \cup (4, 1, 9)$  \hspace{1in} (D6)*$(7,2,4,9,6) \cup (1, 3, 8)$ 
\end{enumerate}

\noindent \textbf{Subcase (A3)-(D1)}. Assume $(7,2,4, 3) \cup (1, 8, 9)$. 
Then  (i)  $(7,6,4,2) \cup (1, 8, 9)$ or  (ii) $(7,6,4,3) \cup (1, 8, 9)$.
See Figure \ref{a3d1}.



\begin{figure}[htpb!]
\begin{center}
\begin{picture}(420, 130)
\put(0,0){\includegraphics[width=5.9in]{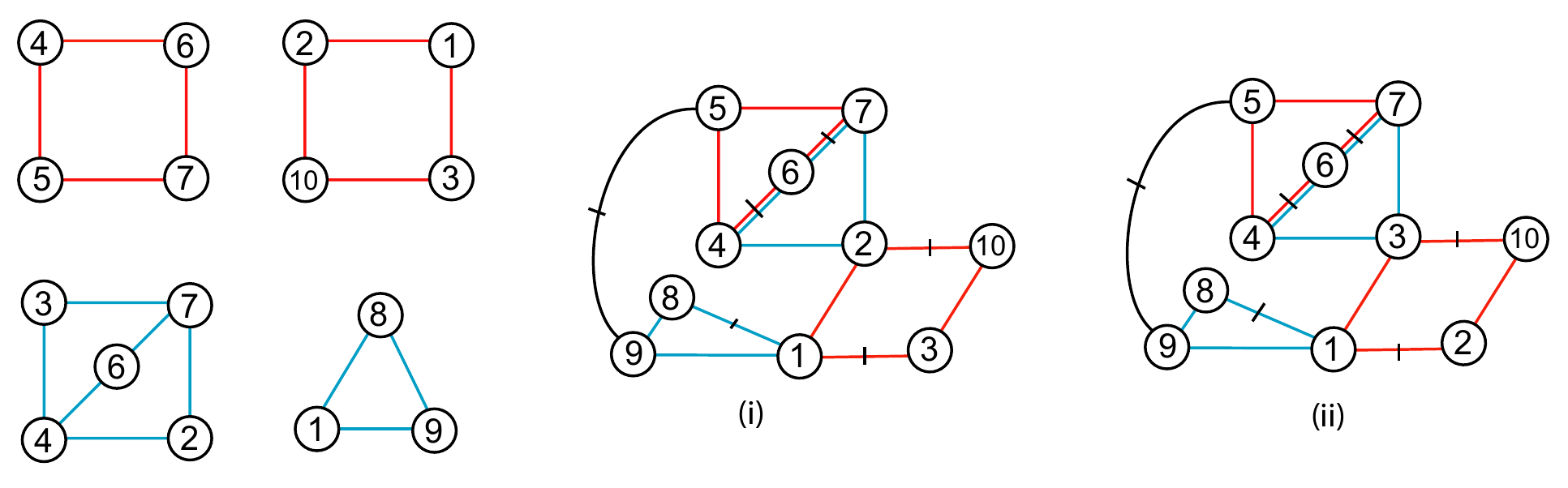}}
\end{picture}
\caption{\small Diagrams for the subcase (A3)-(D1).}
\label{a3d1}
\end{center}
\end{figure} 

\noindent \textbf{Subcase (A3)-(D2)}. Assume $(7,2, 1, 3) \cup (4, 8, 9)$.
Then  (i)  $(7,2,1,5) \cup (4, 8, 9)$ or  (ii) $(7,3,1,5) \cup (4, 8, 9)$.
See Figure \ref{a3d2}.



\begin{figure}[htpb!]
\begin{center}
\begin{picture}(400, 120)
\put(0,0){\includegraphics[width=5.6in]{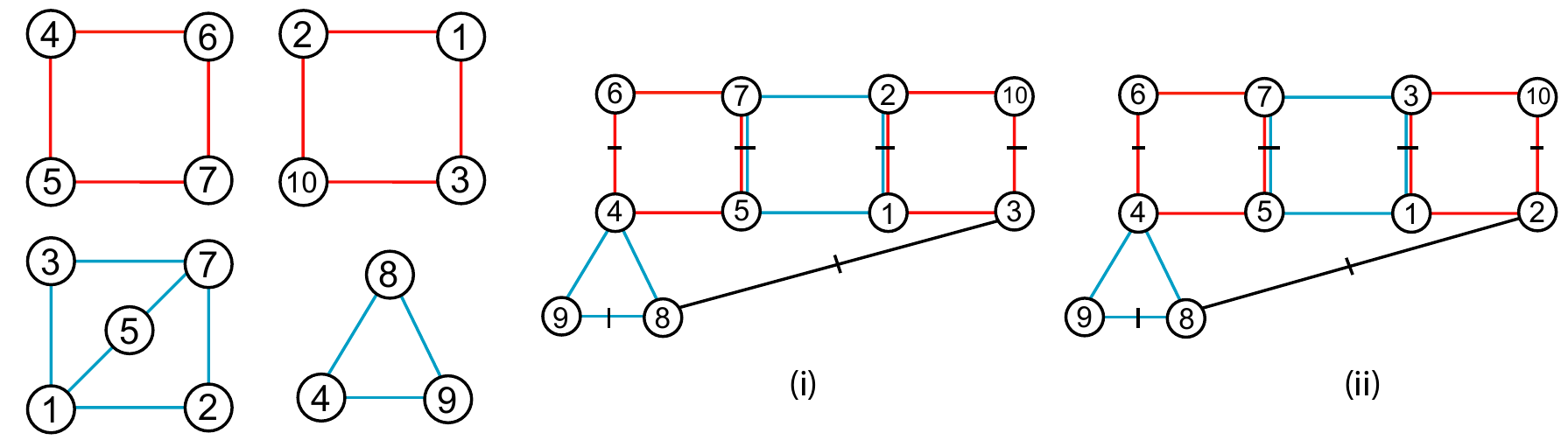}}
\end{picture}
\caption{\small Diagrams for the subcase (A3)-(D2).}
\label{a3d2}
\end{center}
\end{figure} 

\noindent \textbf{Subcase (A3)-(D3)}. Assume  $(7,2,8,3) \cup (4, 1, 9)$. 
Then  (i)  $(7,2,10,3) \cup (4, 1, 9)$ or  (ii) $(8,2,10,3) \cup (4, 1, 9)$.
See Figure \ref{a3d2'}.



\begin{figure}[htpb!]
\begin{center}
\begin{picture}(390, 130)
\put(0,0){\includegraphics[width=5.6in]{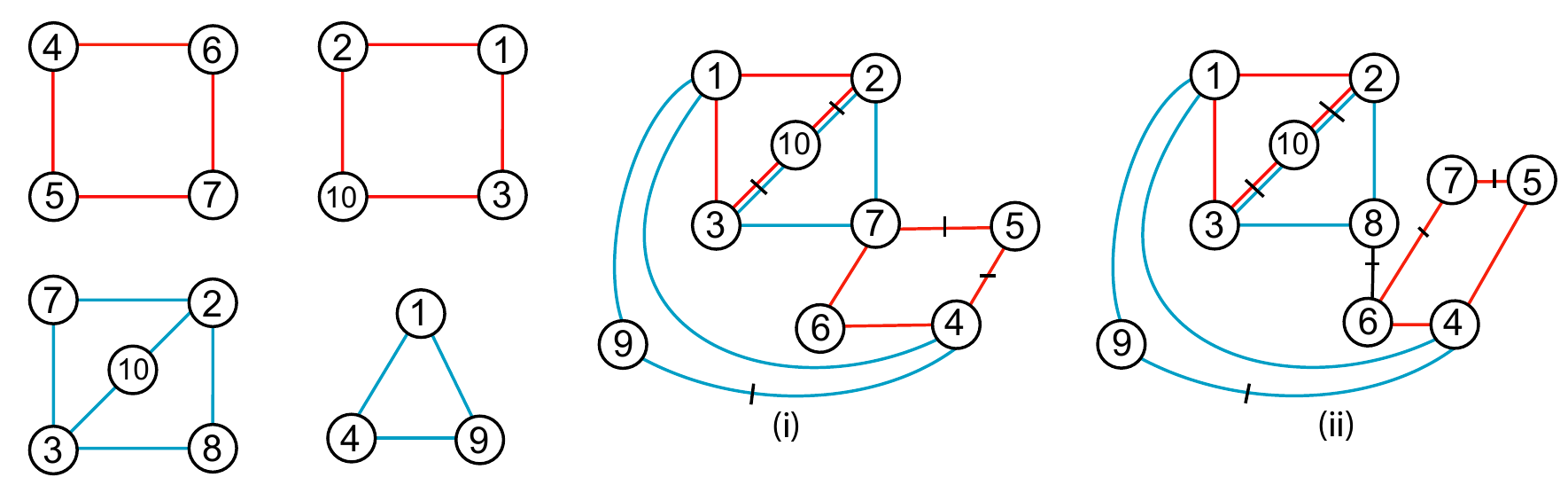}}
\end{picture}
\caption{\small Diagrams for the subcase (A3)-(D3).}
\label{a3d2'}
\end{center}
\end{figure} 

\noindent \textbf{Subcase (A3)-(D4)}. Assume $(7,2,1,9,6) \cup (4, 3, 8)$. 
See Figure \ref{a3d3} (left).



\begin{figure}[htpb!]
\begin{center}
\begin{picture}(340, 80)
\put(0,0){\includegraphics[width=4.5in]{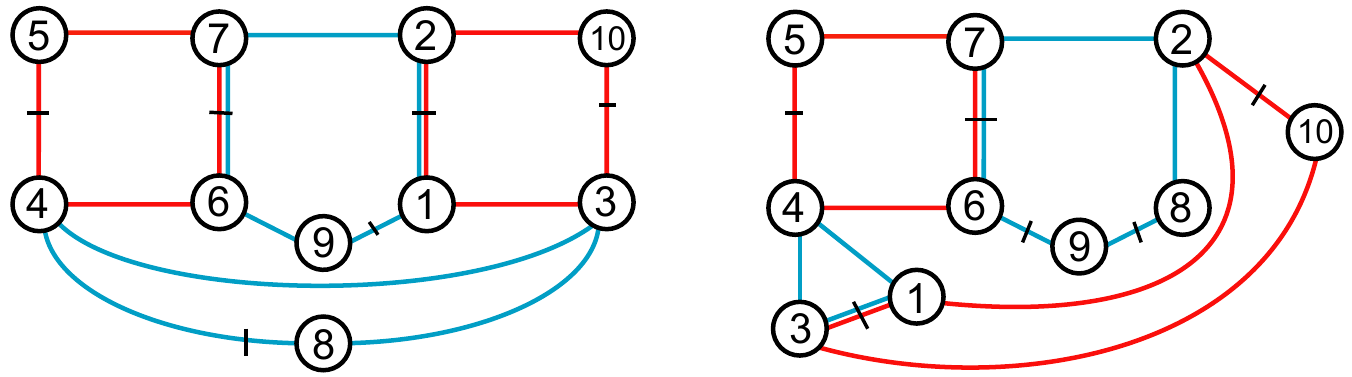}}
\end{picture}
\caption{\small Diagrams for the subcases (A3)-(D4) (left) and (A3)-(D5) (right).}
\label{a3d3}
\end{center}
\end{figure}

\noindent \textbf{Subcase (A3)-(D5)}. Assume $(7,2,8,9,6) \cup (4, 3, 1)$. 
See Figure \ref{a3d3} (right).

If none of the five D-subcases solved above occurs, then there exists a non-trivial link (D6) $(7,2,4,9,6) \cup (1, 3, 8)$.\\

We now match the link in (A3) with each link type of $F$: 
\begin{enumerate}
\item[(F1)] $(5,7,2,10) \cup (3, 1, 9, 6, 8)$ \hspace{0.44in}  (F4) $(5,7,6,10) \cup (2, 1, 3, 8)$
\item[(F2)] $(5,7,2,1) \cup (3, 10, 6, 8)$    \hspace{0.6in}  (F5) $(5,7,6,9,1) \cup (2, 10, 3, 8)$ 
\item[(F3)] $(5,10,2,19) \cup (3, 7, 6, 8)$  \hspace{0.51in}  (F6)*$(5,10,6,9,1) \cup (2, 7, 3, 8)$ 

\end{enumerate}

\noindent \textbf{Subcase (A3)-(F1)}. Assume $(5,7,2,10) \cup (3, 1, 9, 6, 8)$.
See Figure \ref{a3f1} (left).\\
\noindent \textbf{Subcase (A3)-(F2)}. Assume $(5,7,2,1) \cup (3, 10, 6, 8)$.
See Figure \ref{a3f1} (center).\\
\noindent \textbf{Subcase (A3)-(F3)}. Assume $(5,10,2,19) \cup (3, 7, 6, 8)$.
SeeFigure \ref{a3f1} (right).


\begin{figure}[htpb!]
\begin{center}
\begin{picture}(480, 90)
\put(-20,0){\includegraphics[width=6.5in]{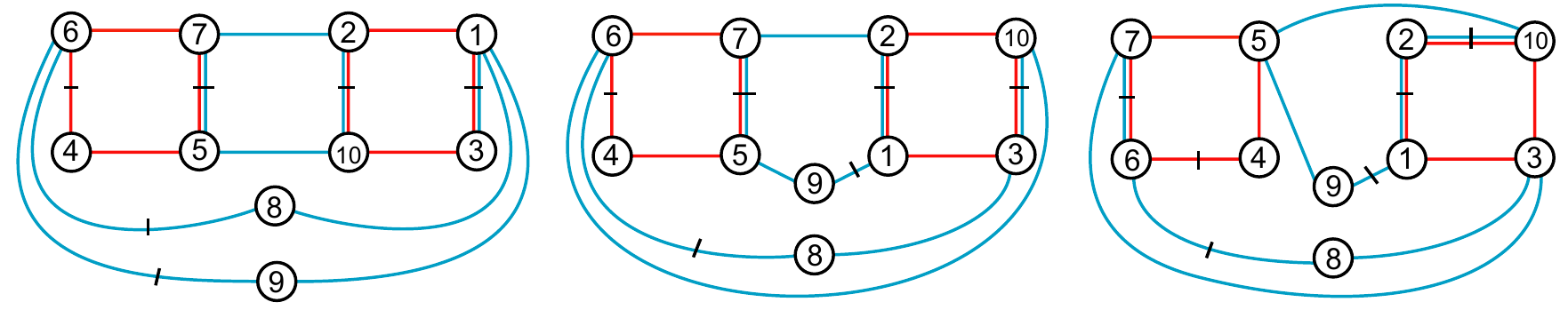}}
\end{picture}
\caption{\small Diagrams for the subcases (A3)-(F1) (left), (A3)-(F2) (center), and (A3)-(F3) (right).}
\label{a3f1}
\end{center}
\end{figure} 

\noindent \textbf{Subcase (A3)-(F4)}. Assume $(5,7,6,10) \cup (2, 1, 3, 8)$.
See Figure \ref{a3f45} (left).
\noindent \textbf{Subcase (A3)-(F5)}. Assume $(5,7,6,9,1) \cup (2, 10, 3, 8)$.
See Figure \ref{a3f45} (center).


\begin{figure}[htpb!]
\begin{center}
\begin{picture}(400, 100)
\put(-20,0){\includegraphics[width=6in]{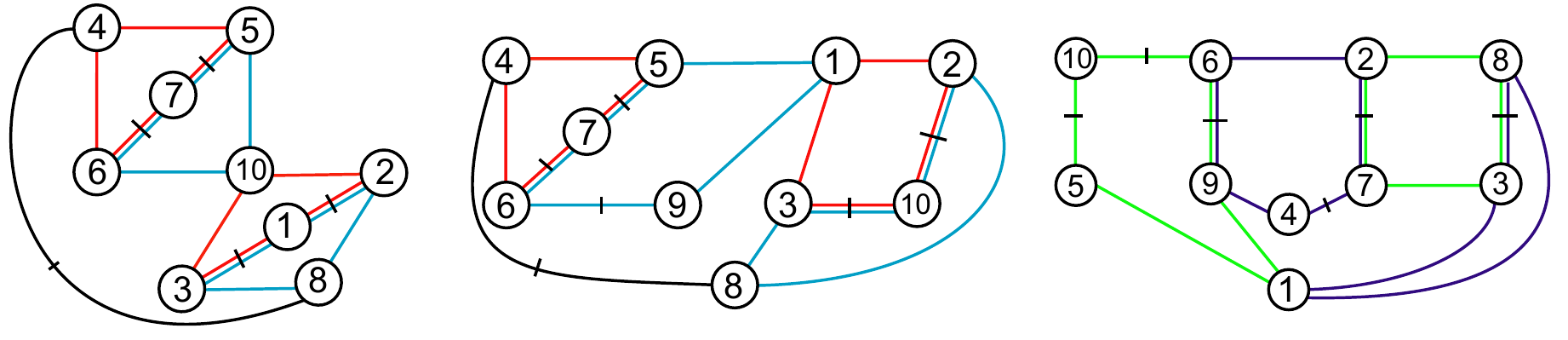}}
\end{picture}
\caption{\small Diagrams for the subcases (A3)-(F4) (left), (A3)-(F5) (center), and  (D6)-(F6) (right).}
\label{a3f45}
\end{center}
\end{figure}

If none of the five F-subcases solved above occurs, then we have (F6) $(5,10,6,9,1) \cup (2, 7, 3, 8)$.
This coupled with the remaining (D6) subcase gives:\\

\noindent \textbf{Subcase (D6)-(F6)}. Assume $(7,2,4,9,6) \cup (1, 3, 8)$ and  $(5,10,6,9,1) \cup (2, 7, 3, 8)$.
See Figure \ref{a3f45} (right).


\textbf{Case (A4)}  Assume $(4,6,7,2) \cup (3, 1, 5, 10)$ is a non-trivial link. 
We look at possible non-trivial links in the graph $B$. 
Based on the symmetries of $\tts$,  $B$ has four  different types of pairs of cycles. 
Since vertices 2 and 3 and vertices 7 and 10, respectively, are distinguished in the link A4, they need to be distinguished within the cycles of B. 
We match the link in (A4) with each link in $B$. 
There is one exceptional case which cannot be solved this way.
Then  we  look at possible non-trivial links in the graph $E$ and we match the link in (A4) with each link in $E$.
There are two  exceptional cases which cannot be solved this way.
We match the two pairs of exceptional cases to complete the proof.

\begin{enumerate}
\item[(B1)] $(8,2,4,3) \cup (7, 5, 10, 6)$     \hspace{0.6in}  (B6) $(8,3,10,6) \cup (4, 5, 7, 2)$ 
\item[(B2)] $(8,2,7,3) \cup (4, 5, 10, 6)$     \hspace{0.6in} (B7)*$(8,2,10,6) \cup (4, 5, 7, 3)$
\item[(B3)] $(8,2,10,3) \cup (4, 5, 7, 6)$     \hspace{0.6in} (B8) $(8,2,4,6) \cup (7, 5, 10, 3)$
\item[(B4)] $(8,2,7,6) \cup (4, 5, 10, 3)$   \hspace{0.6in} (B9) $(8,3,4,6) \cup (7, 5, 10, 2)$
\item[(B5)] $(8,3,7,6) \cup (4, 5, 10, 2)$ 
\end{enumerate}

\noindent \textbf{Subcase (A4)-(B1)}. Assume $(8,2,4,3) \cup (7, 5, 10, 6)$.
See Figure \ref{a4b1} (left).\\
\noindent \textbf{Subcase (A4)-(B2)}. Assume $(8,2,7,3) \cup (4, 5, 10, 6)$.
See Figure \ref{a4b1} (center).\\
\noindent \textbf{Subcase (A4)-(B3)}. Assume  $(8,2,10,3) \cup (4, 5, 7, 6)$ .
See Figure \ref{a4b1} (right).


\begin{figure}[htpb!]
\begin{center}
\begin{picture}(480, 95)
\put(-20,0){\includegraphics[width=6.5in]{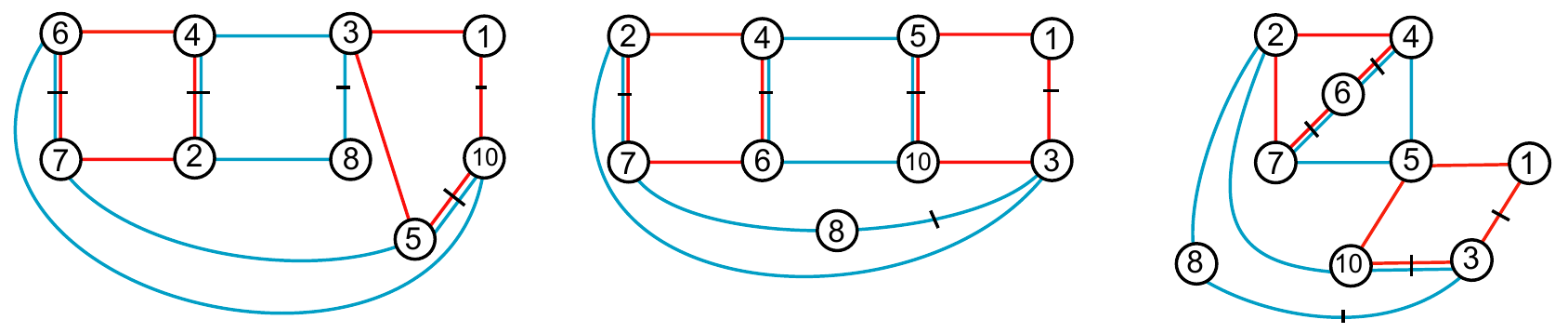}}
\end{picture}
\caption{\small Diagrams for the subcases (A4)-(B1) (left), (A4)-(B2) (center), and  (A4)-(B3) (right).}
\label{a4b1}
\end{center}
\end{figure}

\noindent \textbf{Subcase (A4)-(B4)}. Assume $(8,2,7,6) \cup (4, 5, 10, 3)$.
See Figure \ref{a4b3} (left).\\
\noindent \textbf{Subcase (A4)-(B5)}. Assume $(8,3,7,6) \cup (4, 5, 10, 2)$.
See Figure  \ref{a4b3} (center).\\
\noindent \textbf{Subcase (A4)-(B6)}. Assume $(8,3,10,6) \cup (4, 5, 7, 2)$.
See Figure \ref{a4b3} (right).


\begin{figure}[htpb!]
\begin{center}
\begin{picture}(450, 115)
\put(20,0){\includegraphics[width=5.6in]{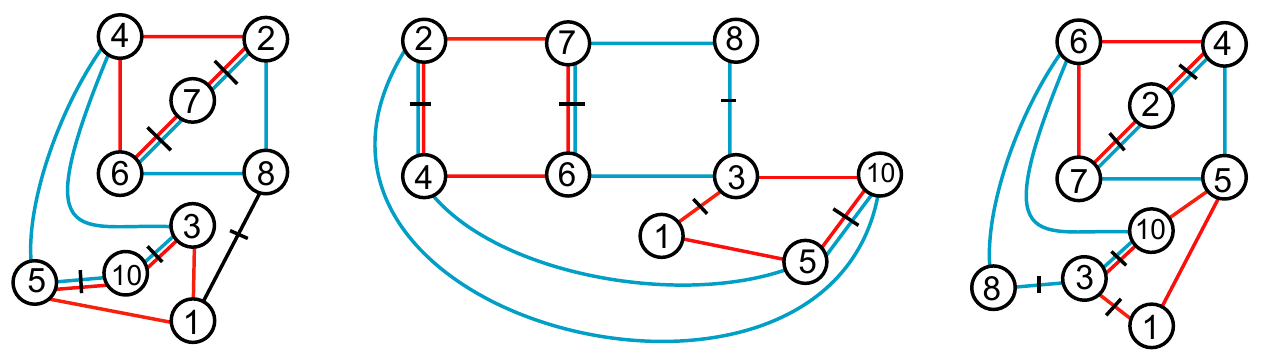}}
\end{picture}
\caption{\small Diagrams for the subcases (A4)-(B4) (left), (A4)-(B5) (center), and  (A4)-(B6) (right).}
\label{a4b3}
\end{center}
\end{figure} 

\noindent \textbf{Subcase (A4)-(B8)}. Assume  $(8,2,4,6) \cup (7, 5, 10, 3)$. See Figure \ref{a4b3} (left).
\noindent \textbf{Subcase (A4)-(B9)}. Assume $(8,3,4,6) \cup (7, 5, 10, 2)$  .
See Figure \ref{a4b3} (right).


\begin{figure}[htpb!]
\begin{center}
\begin{picture}(430, 90)
\put(0,0){\includegraphics[width=6in]{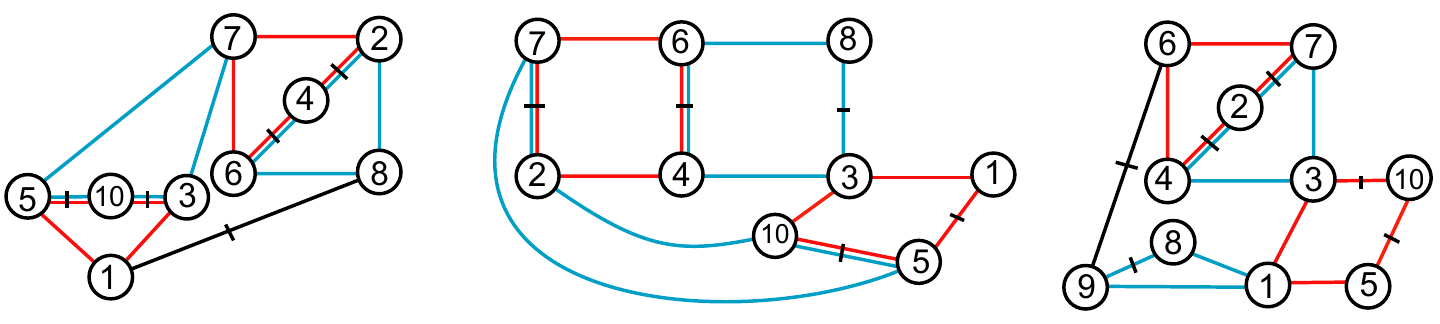}}
\end{picture}
\caption{\small Diagrams for the subcases (A4)-(B8) (left), (A4)-(B9) (center), and  (A4)-(E1) (right).}
\label{a4b4}
\end{center}
\end{figure}

We  look at possible non-trivial links in the graph $E$ and we match the link in (A4) with each link in $E$.

\begin{enumerate}
\item[(E1)] $(7,2,4,3) \cup (1,9,8)$ \hspace{1in} (E6) $(7,2,1,9,5) \cup (4,3,8)$ 
\item[(E2)] $(7,2,1,3) \cup (4, 9,8) $  \hspace{1in} (E7) $(7,2,8,9,5) \cup (4,3,1)$  
\item[(E3)] $(7,2,8,3) \cup (4, 9, 1) $  \hspace{1in} (E8)* $(7,3,8,9,5) \cup (4,2,1)$ 
\item[(E4)] $(7,2,4,9,5) \cup (3, 1, 8)$  \hspace{0.83in} (E9)* $(7,3,1,9,5) \cup (4,2,8)$ 
\item[(E5)] $(7,3,4,9,5) \cup (2, 1, 8)$
\end{enumerate}

\noindent \textbf{Subcase (A4)-(E1)}. Assume  $(7,2,4,3) \cup (1,9,8)$.
See Figure \ref{a4b3} (right).\\
\noindent \textbf{Subcase (A4)-(E2)}. Assume $(7,2,1,3) \cup (4, 9,8) $.
See Figure \ref{a4e2} (left).\\
\noindent \textbf{Subcase (A4)-(E3)}. Assume $(7,2,8,3) \cup (4, 9, 1) $ .
See Figure \ref{a4e2} (center).\\
\noindent \textbf{Subcase (A4)-(E4)}. Assume $(7,2,4,9,5) \cup (3, 1, 8)$.
See Figure \ref{a4e2} (right).


\begin{figure}[htpb!]
\begin{center}
\begin{picture}(420, 90)
\put(-20,0){\includegraphics[width=6.5in]{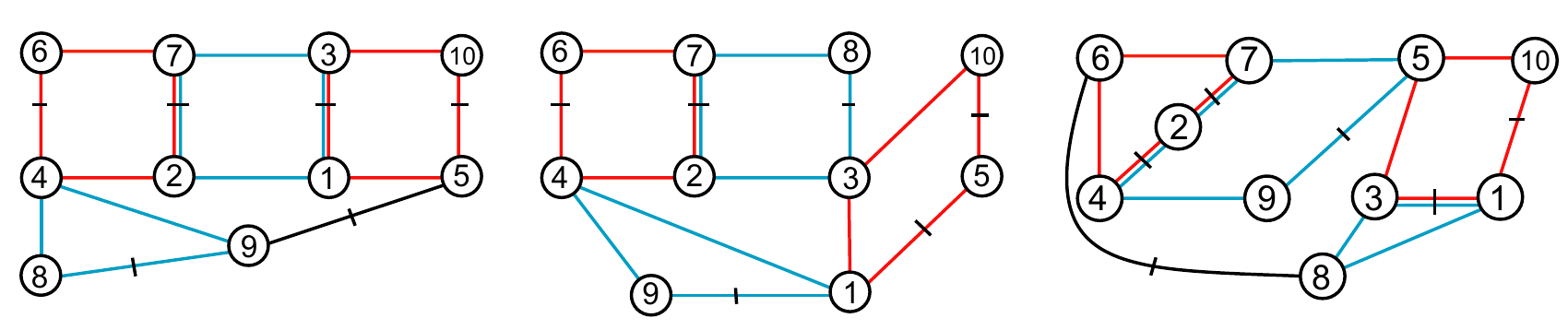}}
\end{picture}
\caption{\small Diagrams for the subcases (A4)-(E2) (left), (A4)-(E3) (center), and  (A4)-(E4) (right).}
\label{a4e2}
\end{center}
\end{figure}

\noindent \textbf{Subcase (A4)-(E5)}. Assume  $(7,3,4,9,5) \cup (2, 1, 8)$.
Then  (i)  $(7,5,10,3) \cup (2, 1, 8)$ or  (ii) $(5,10, 3,4,9) \cup (2, 1, 8)$.
See Figure \ref{a4e3p}. 


\begin{figure}[htpb!]
\begin{center}
\begin{picture}(450, 100)
\put(20,0){\includegraphics[width=5.8in]{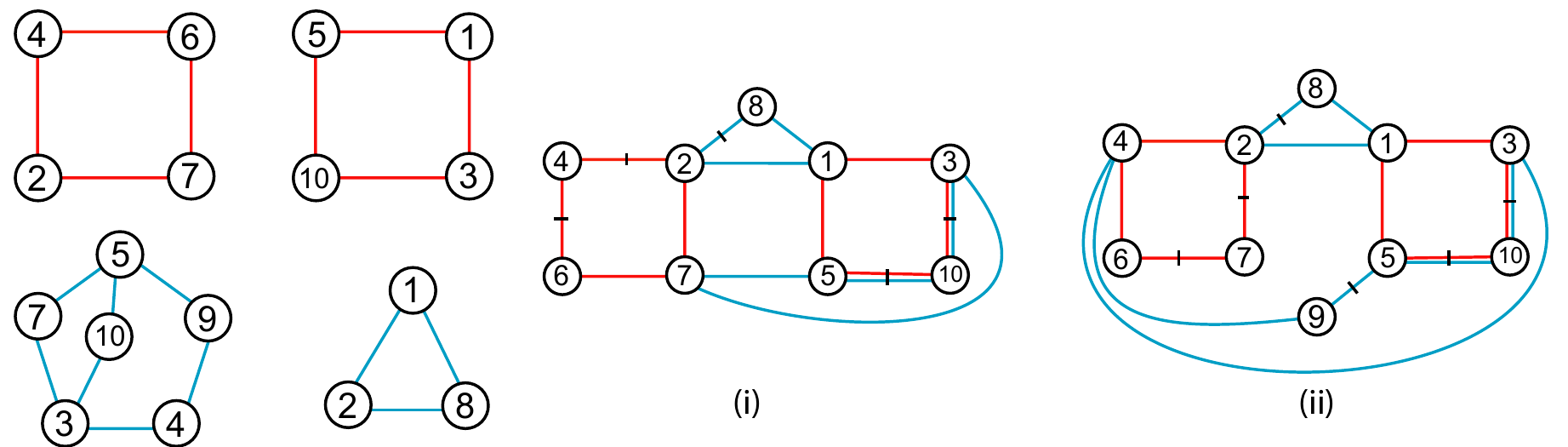}}
\end{picture}
\caption{\small Diagrams for the subcase (A4)-(E5)}
\label{a4e3p}
\end{center}
\end{figure}


\noindent \textbf{Subcase (A4)-(E6)}. Assume $(7,2,1,9,5) \cup (4,3,8)$.
Then  (i)  $(5,7,6,9) \cup (4, 3, 8)$ or  (ii) $(7,6, 9,1,2) \cup (4, 3, 8)$.
See Figure \ref{a4e4}. 
\begin{figure}[htpb!]
\begin{center}
\begin{picture}(420, 110)
\put(20,0){\includegraphics[width=5.5in]{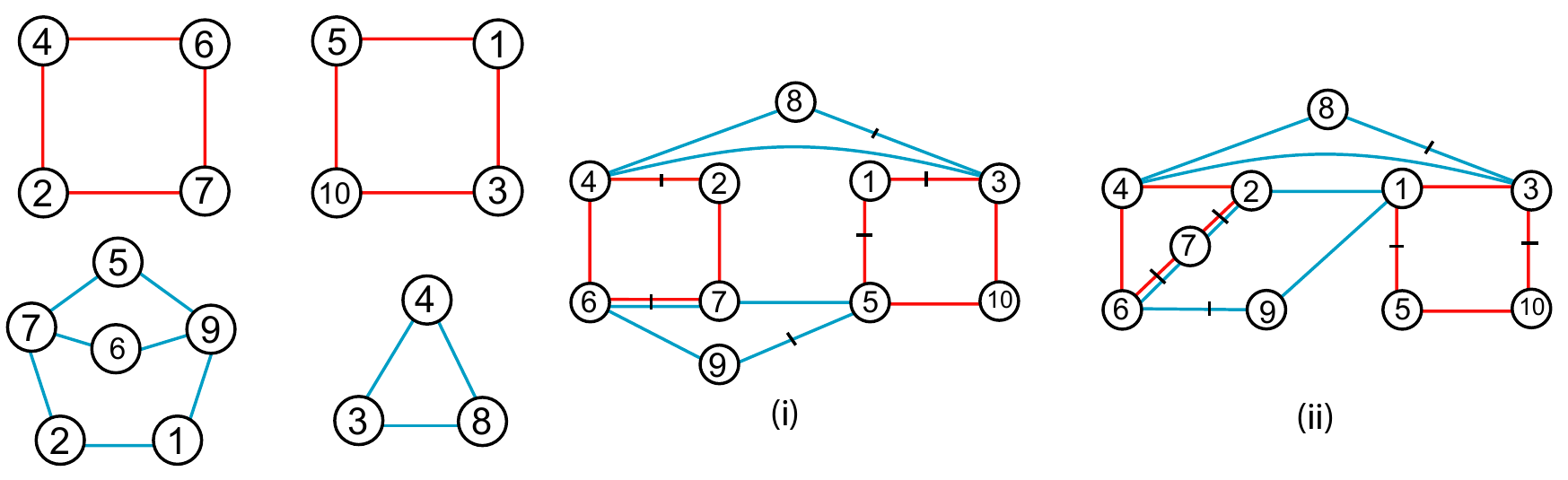}}
\end{picture}
\caption{\small Diagrams for the subcase (A4)-(E6).}
\label{a4e4}
\end{center}
\end{figure}

\noindent \textbf{Subcase (A4)-(E7)}. Assume   $(7,2,8,9,5) \cup (4,3,1)$ .
See Figure \ref{a4e4p} (left).\\
\noindent \textbf{Subcase (B7)-(E8)}. Assume  $(8,2,10,6) \cup (4, 5, 7, 3)$ and  $(7,3,8,9,5) \cup (4,2,1)$. 
See Figure \ref{a4e4p} (center).\\
\noindent \textbf{Subcase (B7)-(E9)}. Assume  $(8,2,10,6) \cup (4, 5, 7, 3)$ and $(7,3,1,9,5) \cup (4,2,8)$.
See Figure \ref{a4e4p} (right).


\begin{figure}[htpb!]
\begin{center}
\begin{picture}(440, 90)
\put(0,0){\includegraphics[width=6.1in]{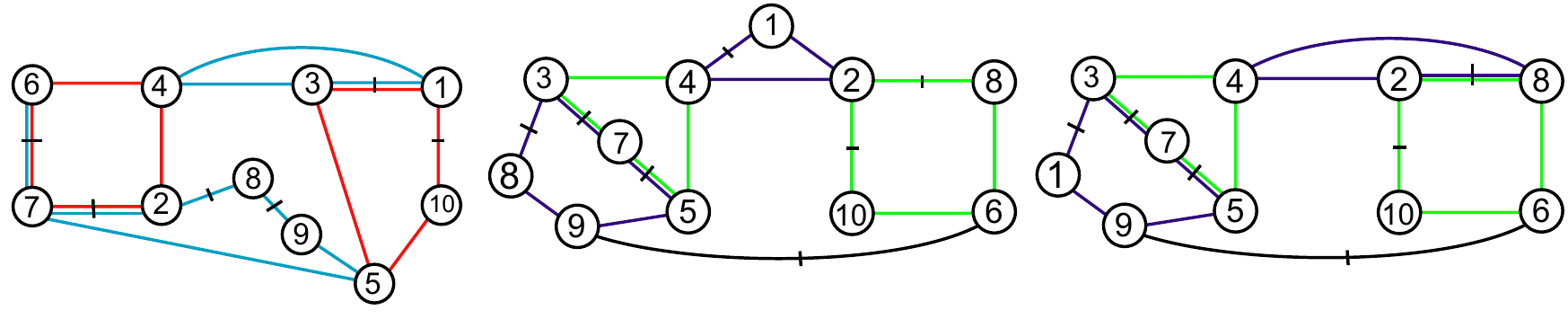}}
\end{picture}
\caption{\small Diagrams for the subcases (A4)-(E7) (left), (B7)-(E8) (center), and  (B7)-(E9) (right).}
\label{a4e4p}
\end{center}
\end{figure} 

\end{proof}


\section{$\tts$ is MMIK}

In this section we prove $\tts$ is MMIK by showing 
every simple minor of it is nIK.
The graph $\tts$ has ten vertices labeled 1, 2, ..., 10. 
Due to the symmetries of the graph, the vertices can be partitioned into six equivalence classes: $\{1,8\}$, $\{2,3\}$, $\{4\}$, $\{5,6\}$, $\{7,10\}$, and $\{9\}$. 
Up to symmetry,  $\tts$ has eleven types of edges.
Representatives for each possible type of edge are  listed in the first column of Table \ref{2Aoperations}.
For each such edge type, we constructed two graphs, one by deleting the edge, and one by contracting the edge.
The graph obtained by deleting the edge is 2-apex, since the removal of the two vertices listed in the second column gives a planar graph.
There is one exception, the graph obtained by deleting the edge $(4,5)$ is not 2-apex.
This graph is shown to be nIK in the next paragraph.
The graph obtained by contracting the edge listed in the first column is 2-apex, since the removal of the two vertices listed in the third column gives a planar graph.
When contracting an edge $e$, the new vertex inherits the smaller label among the endpoints of $e$, and all vertices not incident to $e$ maintain their labels.

\begin{table}[h]
\centering
\label{2Aoperations}
\begin{tabular}{|c|c|c|c|}
\hline
    Edge     & Deletion &   Contraction  \\ \hline
    $\{1,2\}$ & $4,7$ & $1,3$ \\ \hline
        $\{1,4\}$ & $2,6$ & $1,7$ \\ \hline
            $\{1,5\}$ & $2,3$ & $1,2$ \\ \hline
              $\{1,8\}$ & $2,3$ & $1,4$ \\ \hline
                    $\{1,9\}$ & $2,5$ & $2,3$ \\ \hline
                        $\{2,4\}$ & $5,6$ & $2,3$ \\ \hline
                                $\{2,7\}$ & $3,4$ & $2,4$ \\ \hline
                                        $\{4,5\}$ & $*$ & $2,4$ \\ \hline
                                                $\{4,9\}$ & $2,3$ & $4,7$ \\ \hline
                                                    $\{5,7\}$ & $2,4$ & $2,4$ \\ \hline
                                                        $\{5,9\}$ & $2,6$ & $2,5$ \\ \hline

\end{tabular}
\vspace*{.2in}

\caption{\small The graph obtained by deleting the edge in the first column becomes planar when deleting the two vertices in the second column. 
The graph obtained by contracting the edge in the first column becomes planar when deleting the two vertices in the second column. }
\end{table}

The graph $G'$ obtained from $\tts$ by deleting the edges $(4,5)$ is not  2-apex. We show it is nIK.
Denote by $G''$ the graph obtained from $G'$ through a $\ty$-move on the triangle $(1,5,9)$. 
Call the new vertex  11. 
See Figure \ref{Symmetricminus45}.
Delete  vertices 2 and 11 of $G''$ to obtain  a planar graph. 
This proves $G''$ is 2-apex and thus nIK.
Sachs~\cite{Sa} showed that the $\ty$-move preserves intrinsic linking.  
Essentially the same argument shows that the $\ty$-move also preserves intrinsic knotting.
So the graph $G'$ is nIK.

\begin{figure}[htpb!]
\begin{center}
\begin{picture}(440, 150)
\put(0,0){\includegraphics[width=6.3in]{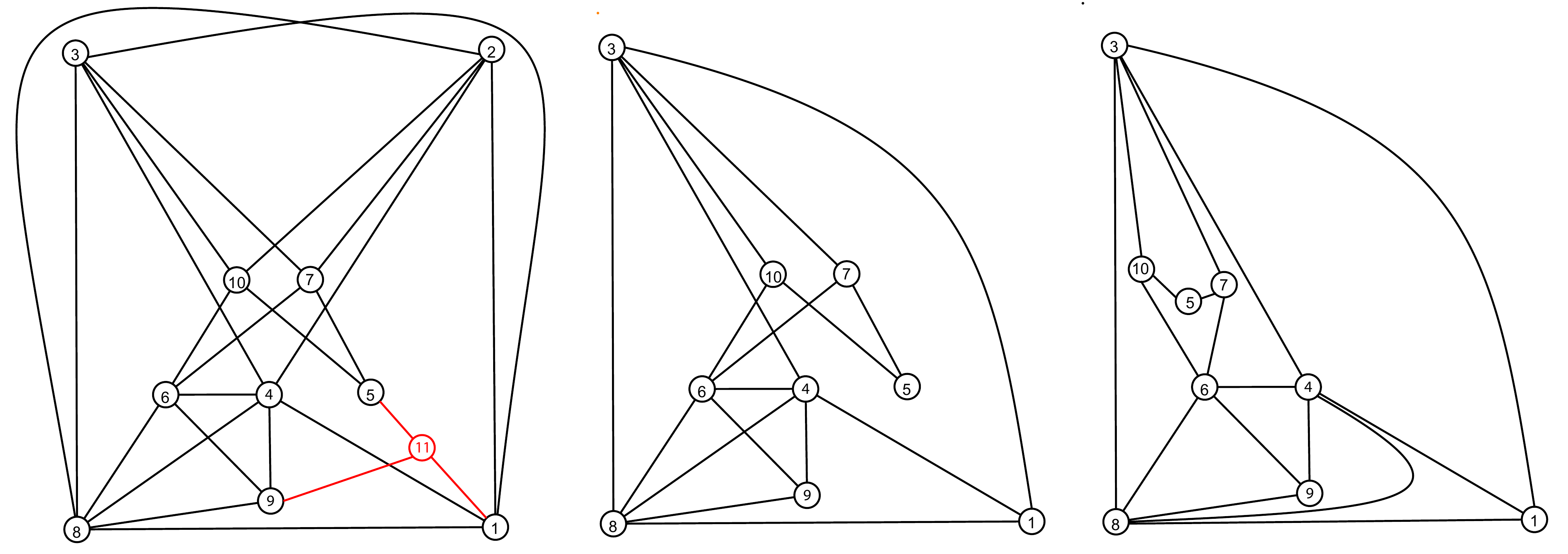}}
\end{picture}
\label{Symmetricminus45}
\caption{\small Left: Graph $G''$ obtained from $\tts$ by removing the edge $(4,5)$ followed by a $\ty$-move on the triangle $(1,5,9)$. Center: Graph $G'''$ obtained from $G''$ by deleting vertices 2 and 11. Right: Planar embedding of $G'''$.}
\end{center}
\end{figure}

\section{$\mu=5$ IK graphs}
\label{sec:mu5IK}

In this section we describe what is known about graphs $G$ with Colin de Verdi\`{e}re invariant five. We begin
with some basic observations. Let $K_1*G$ denote the graph obtained by coning a vertex over $G$, i.e., we 
add a vertex $a$ to $G$ along with edge $av$ for every $v \in V(G)$.

\begin{lemma}\cite{HLS}
\label{lem:basicmu}
Let $G$ be a graph.
\begin{enumerate}
\item If $G$ has at least one edge, then $\mu(K_1 * G) = \mu(G) +1$.
\item If $G'$ is a minor of $G$, then $\mu(G') \le \mu(G)$.
\end{enumerate}
\end{lemma}

\begin{lemma} 
\cite{dV}\cite{HLS}
\label{lem:smallmu}
\begin{enumerate}
\item $\mu(G) \leq 3$ if and only if $G$ is planar.
\item  $\mu(G) \leq 4$ if and only if $G$ is nIL.
\end{enumerate}
\end{lemma}

\begin{lemma} \cite{HLS}
\label{lem:muty}
If $\mu(G) \geq 4$ and a $\ty$ move on $G$ produces
$G'$, then $\mu(G) = \mu(G')$.
\end{lemma}

For $v \in V(G)$, let $G-v$ denote the graph that results after deleting $v$ and
all its edges.

\begin{lemma} 
\label{lem:munapex}
If $G$ is $n$-apex for $n \geq 0$, then $\mu(G) \leq n+3$
\end{lemma}

\begin{proof}
We use induction on $n$. If $n=0$, the result follows from Lemma~\ref{lem:smallmu}.
Suppose $G$ is $n+1$-apex and $v \in V$ such that $G - v$ is $n$-apex. 
Then $G$ is a subgraph of $K_1 * (G-v)$, and, by Lemma~\ref{lem:basicmu},
$\mu(G) \leq \mu(G-v) + 1 \leq (n+1)+3$.
\end{proof}

\begin{lemma} 
\label{lem:muvdel}
If $G$ is IK and there is a vertex $v$ such that $G-v$ is nIL, then $\mu(G) = 5$.
\end{lemma}

\begin{proof}
Robertson, Seymour, and Thomas \cite{RST} established that $G$ IK implies $G$ is IL. By Lemma~\ref{lem:smallmu}, 
$\mu(G) \geq 5$ and $\mu(G-v) \leq 4$. Since $G$ is a subgraph of $K_1 * (G-v)$, using 
Lemma~\ref{lem:basicmu}, $\mu(G) \leq 5$.
\end{proof}

For $e \in E(G)$, let $G-e$ denote the edge deletion minor and $G/e$ the edge contraction minor of $G$.

\begin{lemma} 
\label{lem:musimple}
If $G$ is IK and has a nIL simple minor, then $\mu(G) = 5$.
\end{lemma}

\begin{proof} The proof is similar to that of the previous lemma. In particular $\mu(G) \geq 5$.
By definition, there is an edge $e$ so that $G-e$ or $G/e$ is nIL. Suppose first that $G-e$ is nIL.
By Lemma~\ref{lem:smallmu},  $\mu(G-e) \leq 4$. 
We can form a graph $G'$ homeomorphic to $G$ by adding a degree two vertex between $a$ and $b$, the
vertices of $e$. Then $G'$  is a subgraph of $K_1 * (G-e)$, and using 
Lemma~\ref{lem:basicmu}, $\mu(G') \leq 5$. Since $G$ is a minor of $G'$, by Lemma~\ref{lem:basicmu}, 
$\mu(G) \leq 5$.

Next, suppose $G/e$ is nIL so that $\mu(G/e) \leq 4$. We can again recognize $G$ as a subgraph of
$K_1 * (G/e)$ which implies $\mu(G) \leq 5$.
\end{proof}

We remark that many of the known MMIK graphs have $\mu = 6$. In \cite{FMMNN}, the authors 
provide a listing of 264 MMIK graphs, of which 
105 are in the families of $K_7$, $K_{3,3,1,1}$, and $E_9+e$. We will now verify that each 
of these three graphs has $\mu = 6$. By Lemma~\ref{lem:muty}, all 105 graphs have $\mu$ 
invariant six. As shown in~\cite{dV}, $\mu(K_n) = n-1$ when $n > 1$, so $\mu(K_7) = 6$.
The graph $K_{3,3,1,1}$ is $K_1 * K_{3,3,1}$. Since $K_{3,3,1}$ is an obstruction for
intrinsic linking~\cite{RST}, by Lemma~\ref{lem:smallmu}, $\mu(K_{3,3,1,1}) = \mu(K_{3,3,1}) + 1 \geq 6$. On the other hand,
$K_{3,3,1,1}$ is $3$-apex, which, by Lemma~\ref{lem:munapex}, shows $\mu(K_{3,3,1,1}) \leq 6$.
Since $E_9$ is in the $K_7$ family, by Lemma~\ref{lem:muty}, $\mu(E_9) = \mu(K_7) = 6$. 
By Lemma~\ref{lem:basicmu}, $\mu(E_9+e) \geq \mu(E_9) = 6$. On the other hand,
$E_9+e$ is $3$-apex, so, by Lemma~\ref{lem:munapex}, $\mu(E_9+e) \leq 6$.
By Lemma~\ref{lem:muty} all 110 graphs in the $E_9+e$ family have $\mu = 6$ (not just the
33 that are MMIK). Note that these 110 graphs are all IK~\cite{GMN}.

In contrast, in this paper we have introduced several new examples of IK graphs with
$\mu = 5$. Such examples were known previously. For example, 
Foisy~\cite{F} provided an example of a MMIK graph $F$ that becomes nIL on deletion of 
a vertex. By Lemma~\ref{lem:muvdel}, $\mu(F) = 5$. By Lemma~\ref{lem:musimple},
$\mu(\etf) = 5$ as it is IK with both a nIL edge deletion minor as well a nIL edge contraction minor.
Similarly, $\mu(\tth) = 5$ since it is IK with a nIL edge deletion minor. Finally, 
we argue that $\mu(\tts) = 5$. Since $\tts$ is a minor of $\etf$, then $\mu(\tts) \leq \mu(\etf) = 5$.
On the other hand, as we proved in Section~\ref{Order 10 IK minor}, 
$\tts$ is IK, hence~\cite{RST} IL, and $\mu(\tts) \geq 5$ by Lemma~\ref{lem:smallmu}. 
By Lemma~\ref{lem:muty}, graphs in the families of $\tts$, $\tth$, and $\etf$ also have 
$\mu = 5$. 
Using computers, the $\tts$ family alone provides more than 600 new examples of IK graphs 
with Colin de Verdi{\`e}re invariant five.

We conclude this section with a Proof of Proposition~\ref{prop:order10}.

\begin{proof}
 Assume there exists an IK graph $G$ of order less than 10 which admits a nIL edge contraction minor. 
 As such, by Lemma~\ref{lem:musimple}, $\mu(G) = 5$. 
 Since $\mu$ is minor monotone (Lemma~\ref{lem:basicmu}), any MMIK minor of $G$ must have $\mu=5$. 
 By work of Goldberg, Mattman, and Naimi ~\cite{GMN}, and Mattman, Morris, and Ryker ~\cite{MMR}, the MMIK graphs of order at most 9 are known. 
 With the exception of $G_{9,28}$, depicted in Figure ~\ref{G928}(a), all the others are either in the $K_7$ family, the $K_{3,3,1,1}$ family, or the $E_9$+e family, and thus have $\mu=6$. 
 \begin{figure}[htpb!]
\begin{center}
\begin{picture}(370, 120)
\put(0,0){\includegraphics[width=5in]{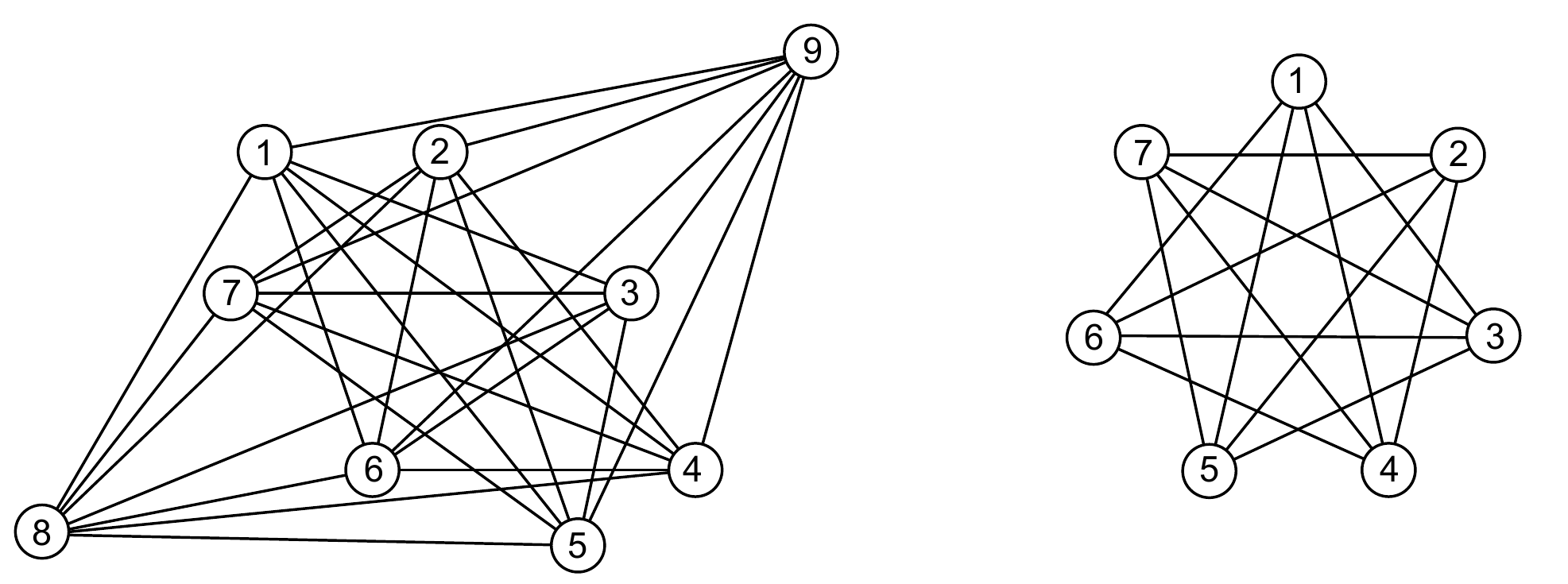}}
\end{picture}
\caption{(a) The graph $G_{9,28}$.\,\,\,\,\,\, (b) The complement of a 7-cycle. }
\label{G928}
\end{center}
\end{figure} 
It follows that $G$ must have order 9 and that $G_{9,28}$ is a subgraph of $G$. 
If contracting an edge $e$ of $G$ produces a nIL minor, then deleting either endpoint of $e$ must also produce a nIL minor (subgraph). 
Since $G_{9,28}$ is a subgraph of $G$, deleting the same vertex must produce a nIL subgraph of $G_{9,28}$. 
The graph $G_{9,28}$ is highly symmetric, having a rich automorphism group, and it is structured as two nonadjacent cones over the complement of a 7 cycle (the graph depicted in Figure ~\ref{G928}(b)). 
Up to isomorphism, there are only two induced subgraphs of order 8 inside $G_{9,28}$: The graph obtained by deleting the vertex labeled 9, and the graph obtained by deleting the vertex labeled 7. 
Neither of these are nIL, since they both have a $K_6$ minor. 
For the first graph, contracting the edges $(4,7)$ and $(2,6)$ produces a complete minor on the 6 vertices. 
For the second graph, contracting the edges $(4,9)$ and $(2,6)$ also produces a complete minor on the 6 vertices.
\end{proof}

\section{Erratum}
\label{erratum}

In this section we discuss an error in the proof of Propostion~2 of \cite{NPS}.
The proposition asserts that if a graph $G$ has a paneled embedding,
and an edge is added to $G$ between two vertices $a$ and $b$
that have a common adjacent vertex $v$,
then $G+ab$ has a knotless embedding.

\begin{figure}[ht]
 \centering
 \includegraphics[width=110mm]{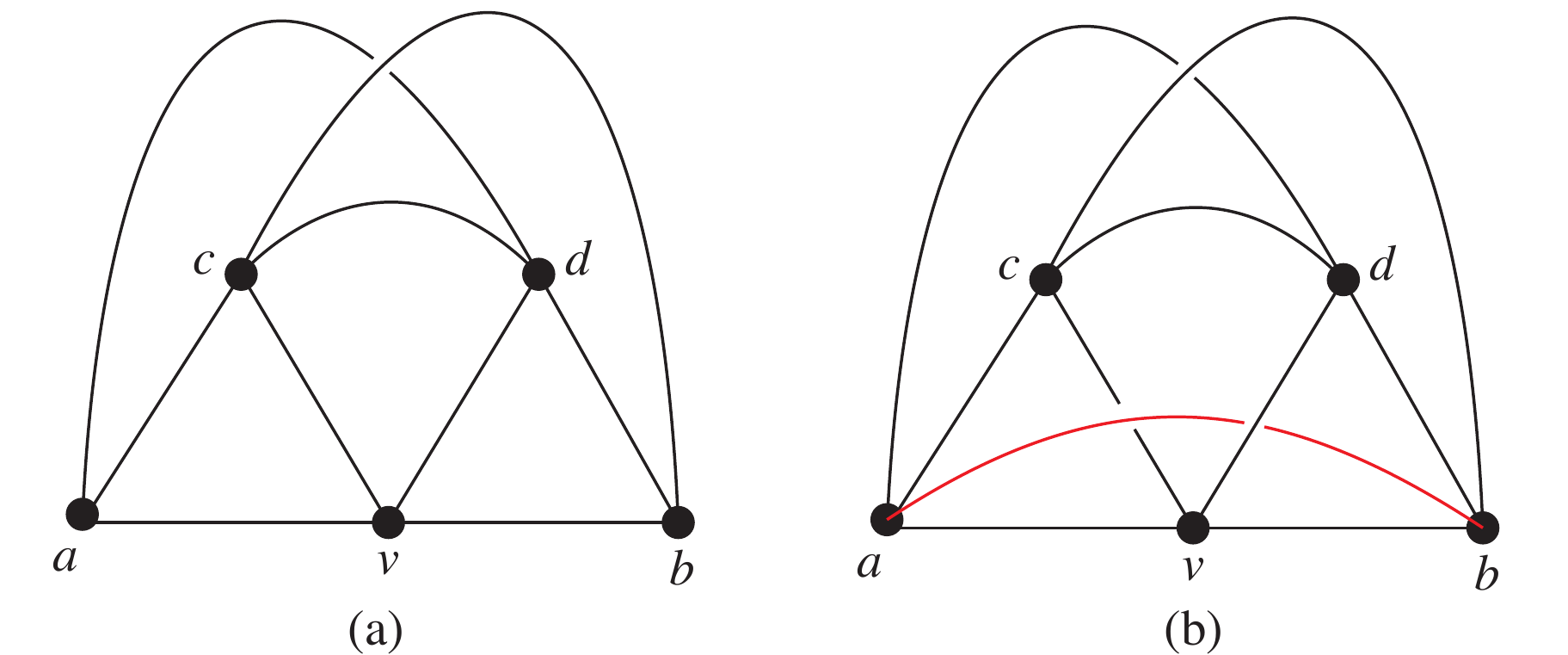}
 \put(-310,80){\large $\Gamma$}
 \put(-150,80){\large $\Gamma'$}

 \caption{
(a) Every cycle in $\Gamma$ is paneled.
(b) $\Gamma'$ contains a trefoil.
 \label {fig-counterexample}
 }
\end{figure}

In the proof of Proposition~2, it is first shown that one can assume 
there is a path $P_{ab} \subset G$ from $a$ to $b$ disjoint from $v$.
Next, the proof claims that
in any paneled embedding $\Gamma$ of $G$,
if $D$ is a panel for the cycle $P_{ab} \cup av \cup vb$ in $\Gamma$,
then embedding the new edge $ab$ in $D$
yields a knotless embedding $\Gamma'$ of $G + ab$.
Figure~\ref{fig-counterexample} shows a counterexample to this claim,
and will be used to explain where the error in the proof of Propostion~2 lies.

It is not difficult to see that in Figure~\ref{fig-counterexample}(a), 
every cycle in $\Gamma$ is paneled.
In particular, the cycle $acdbva$
bounds a panel $D$ such that
$vc$ and $vd$ lie, respectively, below and above $D$ in the figure.
If we embedd the edge $ab$ in $D$ as in Figure~\ref{fig-counterexample}(b),  
we see that the cycle $abcvda$ is a trefoil,
and hence $\Gamma'$ isn't a knotless embedding as claimed.

The error is specifically in the last few sentences of the penultimate paragraph in the proof,
where it mentions a Type~1 Reidemeister move on $P_1 \cup \{e\}$.
The proof overlooks the possibility that 
$P_{bv}$ may prevent this Reidemeister move,
as is the case in Figure~\ref{fig-counterexample}(b)
(for reference, the paths $acdb$, $adv$, and $bcv$ in Figure~\ref{fig-counterexample}
represent, respectively, the paths $P_{ab}$, $P_{av}$, and $P_{bv}$ in the proof of Propostion~2).


\section{Appendix}

We give edge lists for graphs $G_{11,35}, G_{10,30}, $ and $G_{10,26}$:

$$E(G_{11,35})= \{(1, 2), (1, 3), (1, 4), (1, 5), (1, 8), (1, 9), (2, 3), (2, 4), (2,8), (3, 4), (3, 5), $$ $$(3, 6),(3, 7), (3, 8), (3, 10), (3, 11), (4,5), (4, 6), (4, 8), (4, 9), (4, 10), (5, 6), (5, 7), (5, 9), $$ $$ (5, 10), (5, 11), (6, 7), (6, 8), (6, 9), (6, 10), (6, 11), (7, 11), (8, 9), (10, 11), (2, 11)\}$$

$$E(G_{10,30})=\{ (1, 5), (1, 7), (1, 8), (1, 9), (1, 10), (2, 3), (2, 4), (2, 5), (2, 
  6), (2, 7), (2, 10), $$ $$ (3, 4), (3, 6),  (3, 8), (3, 9), (3, 10), (4, 
  6), (4, 8), (4, 9), (5, 6), (5, 7), (5,8), (5,10),  (6, 7), $$ $$ (6, 8), (6, 9), (7, 
  9), (7, 10), (8,10), (9,10) \}$$

$$E(G_{10,26})= \{ (1, 2), (1, 3), (1, 4), (1, 5), (1, 8), (1, 9), (2, 4), (2, 7), (2, 
  8), (2, 10), (3, 4), $$ $$ (3, 7), (3, 8),  (3, 10), (4, 5), (4, 6), (4, 
  8), (4, 9), (5, 7), (5, 9), (5, 10), (6, 7),  (6, 8), (6, 9), $$ $$ (6, 
  10), (8, 9) \}$$

\label{appndx}

\bibliographystyle{amsplain}

\end{document}